\documentclass[12pt]{amsart}
\usepackage{pgfplots}
\usepgfplotslibrary{groupplots}
\usepackage{float}
\usepackage[utf8]{inputenc}
\usepackage{graphicx}
\usepackage{color}
\usepackage{mdwlist}
\usepackage{amssymb}
\usepackage{subfigure}
\usepackage{relsize}
\usepackage{graphics}
\allowdisplaybreaks

\usepackage{cleveref}

\def\g{\lambda}
\def\psiM{\psi_{\max}}

\def\rhoM{\rho_{\max}}
\def\rhom{\rho_{\min}}
\def\betam{\beta_{\min}}
\def\betaM{\beta_{\max}}

\newcommand{\beq}{\begin{eqnarray*}}
\newcommand{\feq}{\end{eqnarray*}}
\newcommand{\beqn}{\begin{eqnarray}}
\newcommand{\feqn}{\end{eqnarray}}

\newcommand{\RN}[1]{%
  \textup{\uppercase\expandafter{\romannumeral#1}}%
}

\newtheorem{theorem}{Theorem}[section]
\newtheorem*{theorem*}{Theorem}
\newtheorem{lemma}[theorem]{Lemma}
\newtheorem{corollary}[theorem]{Corollary}
\newtheorem{proposition}[theorem]{Proposition}
\newtheorem{assump}[theorem]{Assumption}
\theoremstyle{definition}
\newtheorem{definition}[theorem]{Definition}

\theoremstyle{remark}
\newtheorem{remark}[theorem]{Remark}
\numberwithin{equation}{section}
\newcommand{\atan}{\tan^{-1}}

\newcommand\numberthis{\addtocounter{equation}{1}\tag{\theequation}}

\begin{document}
\title[Critical thresholds in Euler-Poisson system]
{A complete characterization of sharp thresholds to spherically symmetric multidimensional pressureless Euler-Poisson systems}% with a constant background}

\author{Manas Bhatnagar and Hailiang Liu}
\address{Department of Mathematics and Statistics, University of Massachusetts Amherst, Amherst, Massachusetts 01003}
\email{mbhatnagar@umass.edu}
\address{Department of Mathematics, Iowa State University, Ames, Iowa 50010}
\email{hliu@iastate.edu} 
%\address{Department of Mathematics, University of South Carolina, Columbia, South Carolina 29208}
% Address of record for the research reported here
%\address{Department of Mathematics, Iowa State University, Ames, Iowa 50010}
%\email{tan@math.sc.edu}
 %35L65 : Conservation laws,  35L67 : Shocks and singularities
%35B30  Dependence of solutions on initial and boundary data, parameters
\begin{abstract}  The Euler-Poisson (EP) system  models the dynamics of a variety of  physical processes, including charge transport, collisional plasmas, and certain cosmological wave phenomena. In this work, we establish sharp critical threshold conditions that distinguish global-in-time regularity from finite-time breakdown for solutions of the radially symmetric, multidimensional pressureless EP system. Overall, there are two cases: with and without background ($c>0, c=0$ respectively). For $c>0$, we obtain precise thresholds assuming a periodicity condition. A key feature of our approach is that it extends seamlessly to the zero background case, where we obtain sharp thresholds without imposing any additional assumptions. In particular, the framework accommodates initial velocities that may be negative, allowing the flow to be directed toward the origin. %In particular, the initial data could include points where velocity is negative, that is, the flow is directed towards the origin. 
 The main analytical challenge of deriving threshold conditions for EP systems stems from the intricate coupling of various  local/nonlocal forces. To overcome this, we identify a novel nonlinear quantity that plays a decisive role in the analysis and enables a unified treatment of all relevant scenarios. Our results provide a comprehensive characterization of critical thresholds for the pressureless EP system in multiple dimensions. 
%helps to analyze the system. %In the case of positive background, if the initial data results in a global-in-time solution, then we show that the density is periodic along any single characteristic path. We use the Floquet Theorem to prove periodicity. 
\end{abstract}
\maketitle

\section{Introduction}
\label{intro}

A general pressureless Euler-Poisson (EP) system has the following form,
\begin{subequations}
\label{genEP}
\begin{align}
& \rho_t + \nabla\cdot(\rho \mathbf{u}) = 0, \quad t>0, \mathbf{x}\in\mathbb{R}^N,
    \label{gensysmass}\\
& \mathbf{u}_t + (\mathbf{u} \cdot \nabla)  \mathbf{u} = -k\nabla\phi , \label{gensysmom}\\
& -\Delta\phi = \rho - c, \label{gensyspoisson}
\end{align}
\end{subequations}
where the initial data $(\rho_0\geq 0, \mathbf{u_0})$ are assumed to be smooth. The constants $k,c\geq 0$ represent the forcing coefficient and background state, respectively. The sign of $k$ signifies the type of particles being modeled and its magnitude gives a measure of the strength between them. When the force in-between particles is repulsive, such as those arising in charge transport, then $k>0$. 
Within the pressureless setup \eqref{genEP}, $k<0$ is relevant in the case of interstellar clouds where the pressure gradient becomes negligible compared to the gravitation forces, see \cite{Ha09}. In the pressureless setup with same charge particles ($k>0$), the background state is, in practicality, a profile, that is, a function of the spatial variable, $c=c(x)$, see \cite{Ja75}. The background models the doping profile for charge flow in semiconductors. %However, the sheer complicacy of the system has restricted researchers to consider the background as a constant. To our knowledge, \cite{BL201} is the only work that considers background as a profile in obtaining critical thresholds.
A recent result in \cite{CKKT25} highlights the necessity of a neutrality condition for one-dimensional pressureless EP systems to ensure global existence. This condition emerges naturally from a rigorous local well-posedness analysis -- filling a gap left open in earlier works.
%The rigorous treatment of the local existence gives rise to a neutrality condition which had been hitherto missing. 
However, for an attractive EP, the neutrality condition enforces the global existence only for a very restrictive (essentially measure-zero) class of initial data. 

A locally well-posed PDE system is said to exhibit a critical threshold phenomena when the global-in-time behavior of its solutions is sensitively dependent on whether the initial data crosses a certain threshold manifold. This manifold divides the phase space of admissible initial data into two disjoint regions. Initial data lying entirely within the \textbf{ subcritical region} yield solutions that remain smooth for all time. In contrast, if any portion of the initial data lies in the \textbf{ supercritical region}, the corresponding solution experiences finite-time breakdown and loses regularity.
%If the initial data lies completely in one of the regions (also called \textbf{subcritical region}), there is global solution. However, if some part of initial data lies outside this set, or in other words, in the \textbf{supercritical region}, a breakdown occurs and solution loses its smoothness in finite time.

As one might expect, establishing global existence and identifying the corresponding critical threshold manifold is considerably simpler in one dimensional case ($N=1$). In this setting, the threshold reduces to a curve in the $(u_{0x},\rho_0)$ plane, and the subcritical region 
can be characterized explicitly as, 
$$
|u_{0x}| < \sqrt{k(2\rho_0-c)}.
$$
A substantial body of literature has been devoted to critical thresholds for system \eqref{genEP} and related models. The existence of such a threshold curve for the Euler–Poisson system was first identified and analyzed in \cite{ELT01} for EP systems, where both the one-dimensional case and the multidimensional case with spherical symmetry were treated. This pioneering work was followed by a series of studies on EP systems and other models, including \cite{BL19,BL201,CCTT16,CCZ16,CT09,HT17,Lee2,LL08,LL09,LT01,LT02, LT03, LT04,TT14, TT22, TW08,Tan20,Tan21,WTB12}, among many others. It is well known that for general hyperbolic conservation laws, singularities form in finite time for arbitrary initial data \cite{Lax64}. This breakdown is driven by the convective nonlinearities inherent in such systems. 
However, the addition of suitable source terms can fundamentally alter this behavior and produce a subset of initial data that leads to global solutions. Identifying this subset, the subcritical region, is an important and intriguing direction of study. In many cases, the “good’’ external forces can balance or even dominate the “bad’’ convective effects, thus expanding the subcritical region.
%This is due to the convective forces present. Addition of source terms however, can result in a set of initial data that leads to global solutions. To extract a set of initial data (subcritical region) for which there is global well-posedness, is an interesting territory to be explored. The `goodness' of the external forces can balance or even outweigh the `bad' convective forces and result in a `large' subcritical region.

A striking example is provided by the strongly singular Euler–Poisson–alignment (EPA) system. In \cite{KT18}, it was shown that system \eqref{genEP} in one dimension, augmented with a nonlocal alignment term in the momentum equation, admits global-in-time solutions for every initial data. This result is remarkable and stands in sharp contrast to the classical behavior of the standard Euler–Poisson equations.

For critical-threshold results in one dimensional EP and EPA systems, we refer to \cite{BL19,BL201,CCTT16,CCZ16,DKRT18,KT18,TW08,Tan21}. In particular, \cite{TW08} considers the Euler–Poisson system with pressure but without background. Deriving thresholds for the full EP system with pressure is notably challenging, due to strict hyperbolicity and the presence of two interacting characteristic families. As a result, the critical-threshold problem for EP systems with pressure and nonzero background is largely an open problem. 
%, the authors include pressure and exclude background. It is quite difficult to obtain thresholds for the full EP system (with pressure) due to the strict hyperbolicity resulting in two characteristic flow paths. Critical threshold for EP systems with pressure and $c>0$ is largely an open problem. 

For pressureless EP systems, the key is to obtain effective bounds on the velocity gradient. 
%Moreover, in the recent work \cite{CKKT25}, the authors show the necessity of a neutrality condition, at least for EP systems in one dimension. Such a treatment for multidimensional EP is absent in literature. It was  
Given the local well-posedness theory, any global-in-time solution to \eqref{genEP} must arise as an extension of a smooth local solution, provided suitable a priori estimates can be established. We recall the relevant local existence result below.
%Owing to the local existence result, global-in-time solutions to \eqref{genEP} can be `extended' upon a priori smooth local solutions. Following is the local existence result.%, see \cite{M84} for a proof. {\color{red}Change the result to the recent one with neutrality condition.}
\begin{theorem}[Local wellposedness]
\label{local}
Consider \eqref{genEP} with smooth initial data, $\rho_0-c\in H^s(\mathbb{R}^N)$ and $\mathbf{u_0}\in \left(H^{s+1}(\mathbb{R}^N)\right)^N$ with $s>N/2$. Then there exists a time $T>0$ and functions $\rho,\mathbf{u}$ such that, 
$$
\rho-c\in C([0,T];H^s(\mathbb{R}^N)),\quad \mathbf{u}\in \left(C([0,T];H^{s+1}(\mathbb{R}^N))\right)^N, 
$$ 
are unique smooth solutions to \eqref{genEP}. In addition, the time $T$ can be extended as long as,
$$
\int_0^T ||\nabla \mathbf{u}(t,\cdot)||_\infty dt < \infty. , 
$$
\end{theorem}
%{\color{red}
\begin{remark}
\label{rem:integrability} We note that, in light of the recent results in \cite{CKKT25} (as alluded to in the first paragraph of this introduction), the solution space must be augmented with an additional  integrability-type condition on $\rho-c$. In our radially symmetric setting, we impose condition \eqref{eq:neutral}, which is sufficient to ensure that the associated ODE dynamics (along characteristics) remain consistent. As demonstrated in \cite{CKKT25}, such a `neutrality' condition is necessary for well-posedness in the one dimensional case. 
However, a precise analogue of this requirement in higher dimensions has not yet been established in the literature. 
%The precise result for higher dimensions is, as of yet, missing in literature.
\end{remark}
%}
From local existence theory, in one dimension, it suffices to control  $|u_x|$ to ensure global well-posedness. However, in higher dimensions, the velocity gradient becomes an $N\times N$ matrix, which introduces additional challenges.  Along the characteristic paths, one can derive an ODE for the divergence of the velocity gradient, but controlling the divergence alone is generally insufficient. To guarantee smooth solutions for all time, it is also necessary to control the spectral gap -- that is, the sum of the absolute differences of the eigenvalues of the gradient matrix (see e.g.,\cite{HT17}). This requirement constitutes the main technical difficulty.

Due to the inherent nature of hyperbolic balance laws, it is generally easier to obtain sufficient conditions for breakdown than for global existence. For example, in \cite{CT08,CT09}, the authors derive bounds on the supercritical region for \eqref{genEP} with $k<0$. Obtaining  bounds on subcritical region is considerably more involved. Various approaches have been proposed by simplifying the EP system.
%Several workarounds have been used by various researchers by simplifying the EP system. 
In particular, \cite{Lee2,LT03} study restricted EP systems in two and three dimensions, respectively, and establish critical thresholds under these simplifications.

Another important variant of the EP system is \eqref{genEP} under spherical symmetry. If the initial data satisfy 
$$
\rho(0,\mathbf{x}) = f(|\mathbf{x}|), \quad \mathbf{u}(0,\mathbf{x}) =g(|\mathbf{x}|)\frac{\mathbf{x}}{|\mathbf{x}|} 
$$
for smooth functions $f$ and $g$, then for as long as the solution remains smooth, this symmetry is preserved in time, i.e., 
$$
\rho(t, \mathbf{x})=\rho(t,|\mathbf{x}|),\quad u(t, \mathbf{x})= u(t,|\mathbf{x}|)\frac{\mathbf{x}}{|\mathbf{x}|}.
$$
Under this simplification, \eqref{genEP} reduces to 
\begin{subequations}
\label{mainsys}
\begin{align}
& \rho_t + \frac{(\rho ur^{N-1})_r}{r^{N-1}} = 0, \qquad t>0,\quad r>0,
    \label{mainsysmass}\\
& u_t + u u_r = -k\phi_r , \label{mainsysmom}\\
& -(r^{N-1}\phi_{r})_r = r^{N-1}(\rho - c), \label{mainsyspoisson}
\end{align}
\end{subequations}
%on a periodic spatial domain  $\mathbb{T} = [-\tfrac12, \tfrac12)$,
with $r=|\mathbf{x}|>0, N\geq 2$ and smooth initial data $(\rho_0\geq 0,u_0)$.
%$$
%\big.(\rho(t,\cdot),u(t,\cdot))\big|_{t=0} = (\rho_0\geq 0,u_0).
%$$
Although this spherically symmetric system is simpler than the full multidimensional EP system, analyzing critical thresholds remains challenging, particularly for $k>0$. 
%Even though the system is now simplified as compared to the general EP system, it turns out that it is still tricky to analyze for thresholds, especially when $k>0$. 
The system \eqref{mainsys} was first studied in \cite{ELT01} for $c=0$ and expanding flows ($u_0> 0$), where sufficient conditions on global existence and finite-time-breakdown were obtained for $N=2,3$, and a sharp threshold condition was obtained for $N=4$. Later, a more general sharp condition was derived in \cite{WTB12}; however, this result  still applied only to zero background and expanding flows. In  \cite{Yu11}, a sufficient condition for finite-time blow-up was derived  for $c=0$, but no corresponding result regarding global existence was obtained.  

The dynamics of \eqref{mainsys} is quite different when $c>0$ compared with $c=0$. One major difference is that with a positive background, the density is treated as a perturbation around the constant background, leading to infinite total mass, and satisfying   
\begin{align}
\label{eq:neutral}
\int_0^\infty r^{N-1}(\rho(t,r) - c) dr = 0,
\end{align}
in contrast to the zero-background case where the total mass is finite and conserved.  

A major advance in relaxing the limitations of the previous results was made in \cite{Tan21}, where the author reduced \eqref{mainsys} to a $4\times 4$ ODE system along a characteristic path and  proved the existence of critical thresholds. This $4\times 4$ ODE system  is crucial in our analysis as well. 

One particularly interesting observation by us was that the Poisson forcing in \eqref{mainsysmom} is sufficient to prevent density concentration at the origin, even if the initial velocity points inward at some locations. To our knowledge, this was the first study of critical thresholds for the spherically symmetric EP system that dropped the expanding-flow assumption.

It was further noted that $N=2$ is critical dimension, requiring a different analytical approach than the case $N\geq 3$. The author also obtained partial results for the case $c=0$, providing bounds on both subcritical and supercritical regions.  However, no results were presented for $c>0$. More recent works \cite{CS23,R23} have highlighted intriguing properties of \eqref{mainsys} with $c>0$. In particular, \cite{CS23} shows that for $c>0,N\neq 4$,  the subcritical region for \eqref{mainsys} is of measure zero in the phase space, while the case  $N=4$ remains an open question. 

In this paper, we develop a unified methodology that addresses both the zero and nonzero background cases. For $c>0$, we use this framework to derive precise critical thresholds in four dimensions,  assuming a periodicity condition. This assumption is motivated by numerical observations as well as the results/discussions in \cite{CS23}. We then apply the same approach to the $c=0$ case, obtaining precise thresholds without any additional assumptions. To our knowledge, such a comprehensive and sharp result is novel and has not appeared in the current literature.
%Such a precise and general result is novel and to our knowledge, not present in current literature. 
 
 Several careful steps are needed to obtain these precise thresholds. 
 In particular, determining the subcritical region for $c>0$ demands special attention. Our strategy begins by characterizing the supercritical region, progressively narrowing down the possible initial configurations that might lead to global solutions, and ultimately isolating the subcritical region.
 %Special care has to be taken with regards to finding the subcritical %region for $c>0$. We essentially start out by characterizing the supercritical region and narrowing down the possible initial configurations that might allow for global solutions, eventually extracting out the subcritical region. 

The methodology we develop is quite general and shows promise for application to radially symmetric Euler-Poisson-alignment systems as well. Hence, as far as possible we perform calculations with a general dimension, $N\geq 2$. However, for $c>0$, we eventually state the final calculations and result for $N=4$. 

%[\red{HL: a summary of main results in several cases is explained here}]
Our main results can be stated in non-technical terms as follows:
\begin{itemize}
\item For the nonzero background case ($c\neq 0$), we show that the EP system, \eqref{mainsys} with $N=4$ and a periodicity assumption admits a global-in-time smooth solution if and only if the initial data is smooth and lies within a subcritical region, denoted by $\Theta_4$. 
The precise critical threshold result is presented in Theorem \ref{ctcn}, with the explicit definition of $\Theta_4$ provided after the result.
\item For the zero background case ($c=0$), we show that the EP system \eqref{mainsys} admits a global-in-time smooth solution if and only if the initial data is smooth and lies within certain subcritical region, $\Sigma_N$. Theorem \ref{c0ctcn} contains the precise thresholds for dimensions greater than or equal to three. Theorem \ref{c0ctc2} contains the threshold results in dimension two. The explicit definition of $\Sigma_N$ is stated after each of the results.
\end{itemize} 
 %Theorem \ref{c0ftbn}, \ref{c0ctcn} and \ref{c0ctc2} pertain to the zero background scenario. 
 We identify a completely novel nonlinear quantity that proves crucial   analyzing the system. At $t=0$, it is defined as    
\begin{align}
\label{thresconst}
A_0(r) := \frac{u_0(r)u_{0r}(r)+ k\phi_{0r}(r)}{r\rho_0(r)},\quad \rho_0(r)>0.
\end{align}
This quantity plays a central role in simplifying the characterization of the subcritical and supercritical regions. Its full motivation and applications will be discussed in detail in the following sections.
Remarkably, using $A_0(r)$, we can derive a more concise breakdown condition, and the corresponding result for dimensions $N\geq 3$ is stated in Theorem \ref{c0ftbn}. 

\subsection{Roadmap for the $c>0$ case} %\red{There is a need to explain main notations and key ideas, so that the main results are more understandable}. 
Before presenting our main results, we give a brief roadmap  outlining how the threshold regions are identified. The key points are summarized as follows: 
\begin{itemize}
\item The full dynamics of \eqref{mainsys} can be reduced to a weakly coupled system of four ODEs along a characteristic path $\{(t,X):dX/dt = u(t,X),\ X(0)=\beta, \beta>0\}$: 
\begin{align*}
& \rho' = -(N-1)\rho q - p\rho, \\
& p' = -p^2 - k(N-1)s + k(\rho -c),\\
& q' = ks-q^2, \\
& s' = -q(c+Ns)
\end{align*}
where
\begin{align*}
p:=u_r,\quad q:=\frac{u}{r},\quad s:=-\frac{\phi_r}{r}.
\end{align*} 
The initial data $\rho_0,p_0,q_0,s_0$ depends on $\beta$; we do not specify it explicitly since our analysis focuses one characteristic path at a time.  This ODE system was first derived in \cite{Tan20}, and a thorough analysis of this ODE system forms the core of our present study. In this paper, we restrict our attention to the case $k>0$, $N\geq 2$ and $c\geq 0$. 
\item The $q-s$ system is decoupled and admits a closed-form representation of trajectory curves, given by, 
$$
 R_N(q(t), \tilde s(t))= R_N(q_0, s_0+c/N),\quad \tilde s = s+c/N
$$
where 
\begin{align}\label{1.4}
R_N(q,\tilde s) = \left\{ 
\begin{array}{c}
\tilde s^{-1}\left( q^2 + \frac{kc}{2} + k\tilde s \ln\left(\tilde s\right) \right),\quad N=2,  \\
 \tilde s^{-\frac{2}{N}}\left( q^2 + \frac{kc}{N} + \frac{2k\tilde s}{N-2} \right),\quad N\geq 3. 
\end{array}
\right.
\end{align}
%\begin{align*}
%& R_2(q, s)= \frac{1}{\left(s+\frac{c}{2}\right)}\left( q^2 + \frac{kc}{2} + k\left( s+\frac{c}{2}\right) \ln\left(s+\frac{c}{2}\right) \right),\quad N=2, \\%\label{trajectory2}\\
%& R_N(q, s)=\left( s+\frac{c}{N}\right)^{-\frac{2}{N}}\left( q^2 + \frac{kc}{N} + \frac{2k\left(s+\frac{c}{N}\right)}{N-2} \right),\quad N>2. %\label{trajectoryN}
%\end{align*}
One can show that the $q-s$ system admits a global bounded solution if and only if $s_0> -c/N$. Moreover, the solutions are periodic and the trajectories rotate clockwise on the $(q, s)$ plane as time progresses. 
Denoting the extremal values of $s$ as $s_{\min}$ and $s_{\max}$, we have 
%$$
%R_N(0, s_{\rm min})=0\;  \text{and}\;  R_N(0, s_{\rm max})=0
%$$
%with 
$$
s_{\rm min} <0< s_{\rm max}.
$$
%where $s_{min},s_{max}$ are the minimum/maximum values of $s$ attained.
In addition, $\int_0^t q(\tau)\, d\tau$ is  bounded,  which implies that 
$$
\Gamma(t):=
e^{-\int_0^t q(\tau)d\tau} = \left(\frac{s(t)+c/N}{ s_0+c/N}\right)^{\frac{1}{N}}\quad > 0.
$$
\item Another key transformation (first used in \cite{Tan21}) 
\begin{align*} 
\eta:= \frac{1}{\rho}\Gamma^{N-1},\qquad w = \frac{p}{\rho}\Gamma^{N-1}
\end{align*}
leads to a new system,
\begin{align*}
& \eta' = w,\\
& w' = -k\eta(c+s(N-1)) + k\Gamma^{(N-1)},
\end{align*}
This is a linear inhomogeneous system with time-dependent but bounded (and in particular, periodic) coefficients.  In this formulation, the existence of a global solution for $(\rho, p)$ is equivalent to ensuring that $\eta (t)>0$ for all $t>0$. 
\item We also define a nonlinear quantity
$$
A=q w-k\eta s,
$$
which is shown to be bounded. Its initial value, $A_0:=q_0w_0-k\eta_0s_0$ and $A_0$ in \eqref{thresconst},  is essentially the same 
quantity in (\ref{thresconst}), and will be discussed further in Section \ref{secblowup}. More precisely, for $N\geq 3$, we have 
$$
A(\Gamma)=\left( A_0 + \frac{k}{N-2} \right)\Gamma - \frac{k}{N-2}\Gamma^{N-1}. 
$$
From this, we can conclude that $\eta(t)$ will reach zero at some positive time if $A$ does not change sign. This occurs if
either 
$$
1+\frac{A_0(N-2)}{k} \leq 0, 
$$
or  
$$
1+\frac{A_0(N-2)}{k} > 0\quad \text{And}\quad \Gamma_{\rm max} \leq \kappa \quad \text{or} \; \Gamma_{\rm min}\geq \kappa,  
$$
%along with, 
%$$
%\Gamma_{\rm max} \leq \kappa \quad \text{or} \; \Gamma_{\rm min}\geq %\kappa, 
%$$
where $\kappa$ is the unique positive root of $A$, 
\begin{align*}
\kappa:= \left( 1+ \frac{A_0(N-2)}{k}\right)^{\frac{1}{N-2}}.
\end{align*} 
\item 
Therefore, to precisely identify the initial configurations that yield  global existence, it is necessary to require 
$$
 1+\frac{A_0(N-2)}{k} > 0 \quad \text{and}\quad  \kappa \in (\Gamma_{\rm min}, \Gamma_{\rm max}).
$$
The key idea is then to construct two bounding functions that demarcate the region where $\eta>0$ for one period of the $q-s$ system. If $\eta$ starts in between these functions at $t=0$, it remains so for all times, except at discrete points where the three functions coincide. Conversely, any $\eta$ outside this region will eventually reach zero.
The two constructed functions effectively form ``beads" around the solution, within which the solution $\eta(t)$ remains, as illustrated in Figure \ref{figmainprop}. More precisely, we show that if $\kappa$ satisfies the above inclusion, there exists two nonnegative functions $\eta_i= \eta_i(t;q_0,s_0,A_0), i=1,2$ such that the following holds:\\
For $q_0\neq 0$, we show that if
$$
\min\{\eta_1(0),\eta_2(0)\}< \eta_0 < \max\{\eta_1(0),\eta_2(0)\},
$$ 
then, 
$$
\min\{\eta_1(t),\eta_2(t)\}\leq \eta(t) \leq \max\{\eta_1(t),\eta_2(t)\},\quad t>0, 
$$ 
The converse also holds, ensuring that $\eta(t)$ stays confined between 
$\eta_1(t)$ and $\eta_2(t)$ for all times.

For $q_0=0$, we have that if
$$
\frac{d\eta_1}{dt}(0)< w_0 < \frac{d\eta_2}{dt}(0), 
$$ 
then similarly,
$$
\min\{\eta_1(t),\eta_2(t)\}\leq \eta(t) \leq \max\{\eta_1(t),\eta_2(t)\},\quad t>0.
$$ 
%Same concern would apply to the case when $q(0)=0$, for which $\eta_1'(0)$ and $\eta_2'(0)$ are required. 
The functions $\eta_1,\eta_2$ are the key components through which the threshold curves for the system are ultimately defined. 

From the results in \cite{CS23}, we know that for $N\neq 4$, the functions $\eta_i$ will eventually cross zero for some future time. However, numerical simulations suggest that for $N=4$, $\eta_i$'s are periodic with the same period as the $q-s$ system.
Assuming this periodicity, we are then able to precisely characterize  the threshold region for the system.
\end{itemize}

\section{Main Results}\label{mainresults}
This section is devoted to presenting precise formulations of our main results. We begin by clarifying the notion of finite-time breakdown for system \eqref{mainsys}. Breakdown manifests in the form
$$
\lim_{t\to t_c^-} |u_r(t,r_c)|=\infty,\quad \lim_{t\to t_c^-} \rho(t,r_c)=\infty\ (or\, 0),
$$
for some $r_c>0$ and finite $t_c>0$. The divergence of $u_r$ (shock formation) follows directly from Theorem \ref{local}.The simultaneous blow-up (or vacuum formation) of the density at the same spacetime point is a consequence of the weak hyperbolicity of the pressureless Euler–Poisson system and is a well-known feature for such systems. In the next section, our analysis will reestablish this behavior within our framework.
%our analysis will allow us to conclude the same. 

We now state our main results, followed by a detailed interpretation. We also provide 
explicit descriptions of the corresponding subcritical region sets. 
\begin{theorem}[Sharp threshold condition]
\label{ctcn}
Consider system  \eqref{mainsys} with $c>0$ and $N = 4$. Assume that assumption \ref{perassump} (periodicity assumption) holds. Then the following dichotomy applies:\\ 
1. {\bf Global-in-time existence.} \\
If for every $r>0$, the initial configuration
satisfies
$$
(r, u_0(r), \phi_{0r}(r), u_{0r}(r),\rho_0(r) )\in\Theta_4,
$$
then the solution exists smoothly for all time. \\
2. {\bf Finite-time breakdown}. \\
If there exists some $r_c>0$ such that,
$$
(r_c,u_0(r_c),\phi_{0r}(r_c),u_{0r}(r_c),\rho_0(r_c))\notin\Theta_4 ,
$$ 
then the corresponding solution undergoes finite-time breakdown.\\
Here, 
%$$
%\Theta_2\subseteq \{ (\alpha,x,y,z,\omega): y/\alpha <c/2, \omega>0 \},
%$$ 
%is as in Definition \ref{deftheta2}, and 
$$
\Theta_4\subseteq \{ (\alpha,x,y,z,\omega): y/\alpha <c/4,  \omega>0 \},
$$ 
is the subcritical region defined precisely in  Definition \ref{defthetan}. 
\end{theorem}

 We describe the subcritical region with the aid of the nonlinear quantity introduced in \eqref{thresconst}. Motivated by that expression, for any point $(\alpha,x,y,z,\omega)$, we define  
 \begin{align}
    \label{acoord}
    & a:=\frac{xz+ky}{\alpha\omega}.
 \end{align}
 This quantity will serve as a key parameter in characterizing the threshold geometry of the phase space.
 \begin{definition}
    \label{defthetan}
    A point $(\alpha,x,y,z,\omega)\in\Theta_4$ if and only if one of the following mutually exclusive conditions holds,
 \end{definition}
\begin{itemize}
    \item $(\alpha,x,y,z,\omega) \in \{(\alpha, 0, 0, z, \omega): z^2 < k(2\omega -c) \},$
    \item For $(x,y)\neq (0,0)$ and $a\in \frac{k}{2}\left(  \sqrt{\frac{y_{m}}{\frac{c}{4} - \frac{y}{\alpha}} } - 1 , \sqrt{ \frac{y_{M}}{\frac{c}{4} - \frac{y}{\alpha}} } - 1 \right), $
\begin{align}
\label{aint}
\begin{aligned}
& \frac{1}{\max\{ \eta_1(0),\eta_2(0) \} } < \omega < \frac{1}{\min\{ \eta_1(0),\eta_2(0) \} }, \quad \text{for } x\neq 0, \\
& z \in \frac{ky}{\alpha a} \left(  \frac{d\eta_1}{dt}(0),  \frac{d\eta_2}{dt}(0)  \right), \quad \text{for } x= 0.
\end{aligned}
\end{align}
\end{itemize}

\iffalse
\begin{definition}
\label{deftheta2}
A point $(\alpha,x,y,z,\omega)\in\Theta_2$ if and only if the following holds,
\end{definition}
\begin{itemize}
    \item For $a\in \frac{k}{2}\left( \ln\left( \frac{ y_{m,2}}{\frac{c}{2}- \frac{y}{\alpha} } \right), \ln\left( \frac{y_{M,2}}{\frac{c}{2}- \frac{y}{\alpha} } \right) \right), $
\begin{align}
\label{aintN2}
\begin{aligned}
& \frac{1}{\max\{ \eta_1(0),\eta_2(0) \} } < \omega < \frac{1}{\min\{ \eta_1(0),\eta_2(0) \} }, \quad \text{for } x\neq 0, \\
& z \in \frac{ky}{\alpha a} \left( \frac{d\eta_1}{dt}(0),  \frac{d\eta_2}{dt}(0)  \right), \quad \text{for } x= 0.
\end{aligned}
\end{align}
\end{itemize}
\fi
Here, $0<y_{m} < \frac{c}{4} - \frac{y}{\alpha} <y_{M}$
%and $0<y_{2,m} < \frac{c}{2} - \frac{y}{\alpha} < y_{2,M}$ 
are the roots of the equation,
%$$
%k\ln(y) + \frac{kc}{2y} = R_2,
%$$
%and
$$
k\sqrt{y} + \frac{kc}{4\sqrt{y}} = R_4^*,
$$
with
$$
R_4^* = R_4\left( \frac{u_0(\alpha)}{\alpha},\frac{c}{4} - \frac{\phi_{0r}(\alpha)}{\alpha} \right).
%\quad R_2 = R_2\left( \frac{u_0(\alpha)}{\alpha}, \frac{\phi_{0r}(\alpha)}{\alpha} \right).
$$
The explicit expression for $R_N(\cdot,\cdot)$ appears in (\ref{1.4}), also in \eqref{Rconstant}. The functions 
$$
\eta_i = \eta_i\left(t; \frac{u_0(\alpha)}{\alpha}, \frac{\phi_{0r}(\alpha)}{\alpha}, A_0(\alpha) \right),\quad i=1,2,
$$ 
are defined via a second order linear initial-value problem and are positive for all except at infinitely many discrete times.

We now give an interpretation of Theorem \ref{ctcn}. As with other critical-threshold results, the theorem allows us to determine pointwise whether the initial data lies in the subcritical region. %Moreover, we can go ahead and essentially construct the whole subcrtical region. 
More importantly, it provides a constructive way to build the subcritical region piece by piece.
%Not only this, we can in fact construct the subcritical region piece by piece.
%Suppose some initial data is given. We proceed to check pointwise and therefore, fix an $r>0$. 
The idea is to divide the full initial data space
in $\mathbb{R}^5$ of $(r, u_0,\phi_{0r},u_{0r},\rho_0)$ into a $3D$ component $(r, u_0,\phi_{0r})$, denoted $P_1$, and a two-dimensional plane $(u_{0r},\rho_0)$, denoted  $P_2$. 

For each point in $P_1$, we can construct the corresponding subricitical region on $P_2$ through $A_0$. 
%We first calculate the corresponding quantity $A_0$ using \eqref{thresconst}. 
To see this, take any point $(\alpha,x,y):=(r,u_0,\phi_{0r})$ in $P_1$. The values $y_{m},y_{M}$ can be calculated directly from $\alpha,x,y$; thus the interval specified in the second bullet point of Definition \ref{defthetan} can now be determined explicitly.  For any (fixed) $a$ in this interval,  equation \eqref{acoord} outputs a line in $P_2(z,\omega)$. In other words, fixing $a$ gives a linear relation between $(u_{0r},\rho_0)$. For this chosen value of $a$, we can then compute the constants, $\eta_1(0)$ and $\eta_2(0)$. Indeed, once  $\alpha,x,y,a$ are fixed, the functions $\eta_i$ become fully determined.  Finally, equation \eqref{aint} provides  the corresponding portion of this line in $P_2$ that contributes to the subcritical region, with the exact form depending on whether $x$ is zero or not. By repeating this for all $a$ in the interval above, we obtain the full portion of the subcritical region associated with the chosen point $(\alpha,x,y)$ in $P_1$. Carrying out this procedure for each point in $P_1$ yields the entire subcritical region.
\begin{figure}[htb]
	\centering
	\begin{tikzpicture}

        \draw [help lines, ->] (-3,0.5) -- (-1, 0.5);
        \draw [help lines, ->] (-3,0.5) -- (-3, 2.5);
        \draw [help lines, ->] (-3,0.5) -- (-4, -0.5);

        \node at (-2,1 ) {\small{$(r,u_0,\phi_{0r})$}};
        \fill (-2,1.4) circle (0.05cm);

        \draw[thick, ->] ( -1.5, 1.7 ) to [out=40, in=160] ( 4, 2 );
        \node at (1,2.1) {Fix $A_0 $  };
        %\node at (1,1.4) {\footnotesize{fixing exactly one of $u_{0r}/\rho_0$. } };

        \draw [help lines, ->] (4.15,0) -- (7, 0);
        \draw [help lines, ->] (4.15,0) -- (4.15, 2.5);

        \node at (3.8,1 ) {\small{$u_{0r}$}};
        \node at (5.5,-0.3 ) {\small{$\rho_{0}$}};

        \draw[thick] ( 4.8, 0.5 ) -- ( 6.2, 1.8 );
        \node[rotate=35] at (4.8,0.5 )   {$\boldsymbol{(}$};
        \node[rotate=35] at (6.2,1.8 )   {$\boldsymbol{)}$};
        \node at (7,1.2 )   {\tiny{ $u_{0r} = \frac{rA_0}{u_0} \rho_0 - \frac{k\phi_{0r} }{u_0} $},};
        \node at (8,0.7 )   {\tiny{ $ \frac{1}{\max\{ \eta_1(0), \eta_2(0) \} } < \rho_0 < \frac{1}{\min\{ \eta_1(0), \eta_2(0) \} } $.}};

	\end{tikzpicture}
	\caption{Illustration of Theorem \ref{ctcn}.}
	\label{fig:illustrate}
\end{figure}

%This follows from the the fact that  depend only on parameters $x,y,a$ as given above, which have been fixed before this step. 
%We can then check whether $(u_{0r},\rho_0)$ satisfy \eqref{aint}  In this way, we are able to check for any given initial data whether it leads to global solutions or finite-time-breakdown. 
%Figure \ref{} illustrates this process.

The next two results provide a complete description of the zero background case in dimensions $N\geq 3$.  %greater than or equal to three.

\begin{theorem}[Sufficient condition for blow-up for $N\geq 3$]
\label{c0ftbn}
Suppose $c=0$ and $N\geq 3$ in \eqref{mainsys}. If there exists an $r_c>0$ such that,
%$$
%u_0(\alpha )u_{0\alpha}(\alpha) < -k\left( \phi_{0\alpha}(\alpha) + \frac{\alpha\rho_0(\alpha)}{N-2} \right),
%$$
%$$
% -k \phi_{0\alpha}(\alpha) > u_0(\alpha )u_{0\alpha}(\alpha) + \frac{k\alpha\rho_0(\alpha)}{N-2},
%$$
$$
A_0(r_c) \leq -\frac{k}{N-2},
$$
then finite-time breakdown occurs. 
\end{theorem}
\begin{remark}
\label{rembifur}
This result offers  a simple criterion for verifying  breakdown. It also highlights the critical nature of the case $N=2$. For $N\geq 3$, there is a flat cutoff that bounds the subcritical region from one side. In contrast, when $N=2$, this flat cutoff disappears:  $A_0(r)$ may take arbitrarily negative values and still remain within the subcritical region. 
\end{remark}

In the next theorem, we describe precisely what happens for initial data satisfying 
$$
A_0(r) >-\frac{k}{N-2} \quad \text{for all}\quad  r>0.
$$
\begin{theorem}[Global solution for $N\geq 3$ with zero background]
\label{c0ctcn}
Suppose $c=0$ and $N\geq 3$ in \eqref{mainsys}. If for every $r>0$, the point 
$$
(r,u_0(r),\phi_{0r}(r),u_{0r}(r),\rho_0(r))
$$
lies in the set 
$$
\Sigma_N\cup \left\{ (\alpha,x,y,z,0): x>0, z\geq -\frac{ky}{x}  \right\},
$$
where 
$$\Sigma_N \subseteq \left\{(\alpha,x,y,z,\omega): y<0,\omega>0,\ \frac{xz+ky}{\alpha\omega} >-\frac{k}{N-2}\right\} $$ 
is the set defined in  Definition \ref{defsigman}, then  solution exists globally in time.  

Moreover, if there exists an $r_c>0$ such that
$$
(r_c,u_0(r_c),\phi_{0r}(r_c),u_{0r}(r_c),\rho_0(r_c))\notin\Sigma_N\cup \left\{ (\alpha,x,y,z,0): x>0, z\geq -\frac{ky}{x}  \right\},
$$ 
then finite-time breakdown occurs. 
\end{theorem}
The quantity $a$ in Definitions \ref{defsigman} and  \ref{defsigma2} is the same parameter that appears in \eqref{acoord}. 
\begin{definition}
\label{defsigman}
A point $(\alpha,x,y,z,\omega)\in\Sigma_N$ (for $N\geq 3$) if and only if one of the following holds:
\end{definition}
\begin{itemize}
    \item Case $a\in \left( -\frac{k}{N-2},0 \right)$: 
    \begin{align}
    \label{aneg}
    \begin{aligned}
    & z > -\frac{ky}{x} + \frac{\alpha a}{x\eta_1(0)},\quad  x\neq 0, \\
    %w> \frac{1}{\eta_1(0)},\quad \text{for } x<0, 
    %& z >  -\frac{ky}{x} + \frac{\alpha a}{x\eta_{1}(0)},\quad \text{for } %x>0,\\
    %w < \frac{1}{\eta_{11}(0)},\quad \text{for } x>0,
    & z > \frac{ky}{\alpha a}\frac{d\eta_{1}}{dt}(0), \quad   x=0.
    \end{aligned}
    \end{align}
    \item Case $a=0$: 
    \begin{align}
    \label{azero}
    \begin{aligned}
    & z=-\frac{ky}{x},\ \omega>\frac{1}{\eta_2(0)},\quad  x<0,\\
    %& \left\{(\alpha,x,y,z,w): xz+ky=0\right\},\quad \text{for } x>0.
    & z=-\frac{ky}{x},\ \omega>0 ,\quad  x>0.
    \end{aligned}
    \end{align}
    \item Case $a\in \left(0, \frac{k}{N-2}\left( \left(\frac{-\alpha y_{N}}{y} \right)^{1-\frac{2}{N}}-1\right)\right)$: 
    \begin{align}
    \label{apos1}
    \begin{aligned}
    & -\frac{ky}{x} + \frac{\alpha a}{x\eta_1(0)} < z < -\frac{ky}{x} + \frac{\alpha a}{x\eta_2(0)},\quad  x<0,\\
    %& \left\{(\alpha,x,y,z,w): \frac{xz+ky}{\alpha w}=a\right\},\quad \text{for } x>0.
    & \omega > 0,\quad  x>0.
    \end{aligned}
    \end{align}
    \item Case $a\in \left[ \frac{k}{N-2}\left( \left(\frac{-\alpha y_{N}}{y} \right)^{1-\frac{2}{N}}-1\right), \infty\right)$: 
    \begin{align}
    \label{apos2}
    \begin{aligned}
    %& \{\emptyset\},\quad \text{for } x<0,\\
    %& \left\{(\alpha,x,y,z,w): \frac{xz+ky}{\alpha w} = a\right\},\quad \text{for } x>0.
    & \omega > 0,\quad  x>0.
    \end{aligned}
    \end{align}
\end{itemize}
Here the parameter $y_N$ satisfies  
$$
0< -\frac{y}{\alpha} < y_{N} 
%= y_{N}\left( \frac{u_0(\alpha)}{\alpha}, \frac{\phi_{0\alpha}(\alpha)}{\alpha} \right) 
:= \left( \frac{(N-2)}{2k} \right)^{\frac{N}{N-2}}\left(R_N\left( \frac{u_0(\alpha)}{\alpha}, \frac{-\phi_{0\alpha}(\alpha)}{\alpha} \right)\right)^{\frac{N}{N-2}}.
$$
The explicit expression for $R_N(\cdot,\cdot)$ is given in \eqref{Rconstant} (with $c=0$). The functions 
$$
\eta_i = \eta_i\left(t; \frac{u_0(\alpha)}{\alpha}, \frac{\phi_{0r}(\alpha)}{\alpha}, A_0(\alpha) \right),\quad i=1,2,
$$ 
are defined later via linear initial value problems.

\begin{theorem}[Global solution for $N=2$ with zero background]
\label{c0ctc2}
Suppose $c=0$ and $N=2$ in equation \eqref{mainsys}. If for every $r>0$, the point
$$
(r,u_0(r),\phi_{0r}(r),u_{0r}(r),\rho_0(r))
$$
belongs to 
$$
\Sigma_2\cup  \left\{ (\alpha,x,y,z,0): x>0, z\geq -\frac{ky}{x}  \right\},
$$
where $\Sigma_2 \subseteq \left\{(\alpha,x,y,z,\omega): y<0,\omega>0\right\}  $ is defined in Definition \ref{defsigma2}, then the corresponding solution exists globally in time.
%there is global solution. 

Conversely, if there exists some $r_c>0$ such that
$$
(r_c,u_0(r_c),\phi_{0r}(r_c),u_{0r}(r_c),\rho_0(r_c))\notin\Sigma_2\cup \left\{ (\alpha,x,y,z,0): x>0, z\geq -\frac{ky}{x}  \right\},
$$ 
then solution experiences finite-time breakdown.
\end{theorem}

\begin{definition}
\label{defsigma2}
A point $(\alpha,x,y,z,\omega)\in\Sigma_2$ if and only if one of the following holds:
\end{definition}
\begin{itemize}
    \item Case $a\in \left(-\infty,0 \right)$:
    \begin{align}
    \label{anegN2}
    \begin{aligned}
    & z > -\frac{ky}{x} + \frac{\alpha a}{x\eta_1(0)},\quad  x\neq 0, \\
    %w> \frac{1}{\eta_1(0)},\quad \text{for } x<0, 
    %& z >  -\frac{ky}{x} + \frac{\alpha a}{x\eta_{1}(0)},\quad \text{for } %x>0,\\
    %w < \frac{1}{\eta_{11}(0)},\quad \text{for } x>0,
    & z > \frac{ky}{\alpha a}\frac{d\eta_{1}(0)}{dt}, \quad x=0.
    \end{aligned}
    \end{align}
    \item Case $a=0$: 
    \begin{align}
    \label{azeroN2}
    \begin{aligned}
    & z=-\frac{ky}{x},\ \omega>\frac{1}{\eta_2(0)},\quad  x<0,\\
    %& \left\{(\alpha,x,y,z,w): xz+ky=0\right\},\quad \text{for } x>0.
    & z=-\frac{ky}{x},\ \omega>0,\quad  x>0.
    \end{aligned}
    \end{align}
    \item Case $a\in \left(0, \frac{k}{2}\ln\left( \frac{-\alpha y_{2}}{y} \right) \right)$:
    \begin{align}
    \label{apos1N2}
    \begin{aligned}
    & -\frac{ky}{x} + \frac{\alpha a}{x\eta_1(0)} < z < -\frac{ky}{x} + \frac{\alpha a}{x\eta_2(0)},\quad  x<0,\\
    %& \left\{(\alpha,x,y,z,w): \frac{xz+ky}{\alpha w}=a\right\},\quad \text{for } x>0.
    & \omega >0,\quad x>0.
    \end{aligned}
    \end{align}
    \item Case $a\in \left[\frac{k}{2}\ln\left( \frac{-\alpha y_{2}}{y} \right), \infty\right)$:
    \begin{align}
    \label{apos2N2}
    \begin{aligned}
    %& \{\emptyset\},\quad \text{for } x<0,\\
    %& \left\{(\alpha,x,y,z,w): \frac{xz+ky}{\alpha w} = a\right\},\quad \text{for } x>0.
    & \omega >0,\quad  x>0.
    \end{aligned}
    \end{align}
\end{itemize}
Here, 
$$
0<-\frac{y}{\alpha}<y_{2} 
%= y_{2}\left( \frac{u_0(\alpha)}{\alpha}, \frac{\phi_{0\alpha}(\alpha)}{\alpha} \right) 
:= e^{\frac{R_2\left( \frac{u_0(\alpha)}{\alpha}, \frac{-\phi_{0\alpha}(\alpha)}{\alpha} \right)}{k}}.
$$
The explicit expression for $R_2(\cdot,\cdot)$ is given in \eqref{Rconstant} (with $c=0$).

%\begin{figure}[h!] 
%\centering
%\subfigure[Weak alignment ($\lambda = \sqrt{2}, \psiM = 0.75,\psim = 0.25 $)]{\label{subcrGrhoweakinitial} \includegraphics[width=0.4\linewidth]{sigma1.jpg}}%
%\qquad
%\subfigure[Strong alignment ($\lambda = \sqrt{2}, \psiM = 2,\psim = 1.5 $)]{\label{subcrGrhostronginitial}\includegraphics[width=0.4\linewidth]{sigma2.jpg}}
%\caption{Shapes of $\Sigma_1,\Sigma_2,\Delta_1,\Delta_2$.}
%\label{mainfig}
%\end{figure}

%\begin{figure}[h!] 
%\centering
%\subfigure{\includegraphics[width=0.4\linewidth]{l1subcr.jpg}}
%\caption{Shape of a subcritical region when $\lambda = 4, ||\psi||_{L^1} = 2, \gamma=0.95$.}
%\label{l1subcrfig}
%\end{figure}

The rest of this paper is arranged as follows. %Section \ref{mainresults} contains the statements of the main results in this paper. 
Section \ref{prelim} presents preliminary calculations  
showing how the full dynamics can be reduced to a weakly coupled system of four ODEs along characteristics. 
Section \ref{secqs} focuses on two of these four ODEs, that are decoupled from the remaining equations, and obtains a closed form of trajectory curves with related properties established for later use.
Section \ref{secblowup} entails conditions under which the velocity gradient or density experiences finite-time blow up. Section \ref{secglobal} constructs the precise subcritical region leading to global-in-time  solutions and proves Theorem 2.1. The zero-background case is analyzed in Section \ref{secc0}. Finally,  Section \ref{secconc} provides concluding remarks.  The code for  generating numerical figures in this paper is available at https://gitlab.com/mbhatnagar011/pressureless-sphericalsymm-ep.

\section{Preliminary calculations}
\label{prelim}
From Theorem \ref{local}, we observe that extending local solutions requires obtaining  suitable bounds on the velocity gradient. The following result, originally derived by the authors in \cite{TT22}, identifies the specific quantities that must be controlled for system  \eqref{mainsys}. 
%For the sake of completion, we also provide the proof here.
\begin{lemma}
\label{radialfield}
Let $f(\mathbf{x})=g(r)$ be a radially symmetric function. Then,  for $\mathbf{x}\neq 0$, the Hessian matrix $\nabla^2f$ has exactly two eigenvalues: $g_{rr}$ and $g_r/r$. Moreover, the vector $x$ is an eigenvector associated with the eigenvalue $g_{rr}$,while the orthogonal subspace $x^\perp$ forms the eigenspace corresponding to the eigenvalue $g_r/r$. 
%for the other eigenvalue. 
\end{lemma}
Though this lemma was previously proved in \cite{TT22}, we include the proof here for the sake of completion.
\begin{proof}
Let $r=|x|$. Taking the gradient of $f$, we obtain 
$$
\nabla f = g_r\frac{x}{r}.
$$
Taking gradient again,
$$
\nabla^2 f = \left( g_{rr} - \frac{g_r}{r}\right) \frac{x\otimes x}{r^2} + \frac{g_r}{r} \mathbb{I}. 
$$
Note that $(\nabla^2 f) x = g_{rr} x$ and $(\nabla^2 f) v = \frac{g_r}{r}v$ for any $v\in x^\perp$. This completes the proof of the Lemma.
\end{proof}
Using this Lemma, we conclude that the Hessians of radial functions are diagonalizable with the same orthonormal basis of eigenvectors. From Theorem \ref{local}, we know that in order for solutions to persist for all time, the velocity gradient matrix must remain bounded. Since the velocity field is radial, we have the following relation for $\mathbf{u}$, $u$ in \eqref{gensysmom} and \eqref{mainsysmom}, respectively, 
\begin{align*}
& \mathbf{u} = u\frac{\mathbf{x}}{r}.
\end{align*}
%Consequently, the Hessian of $v$ is the gradient matrix of velocity. 
Hence, the two eigenvalues of $\nabla\mathbf{u}$ will be $u_r$ and $u/r$.
%$u_r$ must be bounded. Owing to this, if we have $u = g_r$, then $Df = g_r x/r = u x/r$. Hence, gradient of velocity is indeed the Hessian of $f$ and 
Therefore, to ensure that the velocity gradient is bounded, we must control the two quantities: $u_r, u/r$. 

Inspired by this observation, we derive ODEs along the characteristic paths, %We will reduce \eqref{mainsys} to a system of ODEs along the characteristic,
\begin{align}
\label{chpath}
& \left\{ (t,X(t)): \frac{dX}{dt} = u(t,X),\ X(0) = \beta \right\},
\end{align}
for these key quantities.
Rearranging  equation \eqref{mainsysmass}, we obtain,
\begin{align}
\label{rhoeq}
& \rho_t + (\rho u)_r = -(N-1)\rho\frac{u}{r}.
\end{align}
Next, observe that by \eqref{mainsyspoisson}, 
$$
-\phi_{rr} = (N-1)\frac{\phi_r}{r} +\rho-c.
$$
Taking spatial derivative of \eqref{mainsysmom} and using the  equation above yields
\begin{align*}
u_{rt} + uu_{rr} + u_r^2 & = -k\phi_{rr}\\
& = k(N-1)\frac{\phi_r}{r} + k(\rho -c). \numberthis \label{peq}\\
\end{align*}
Next using \eqref{mainsysmom},
\begin{align*}
\left(\frac{u}{r}\right)_t + u\left(\frac{u}{r}\right)_r & = \frac{u_t+uu_r}{r} - \frac{u^2}{r^2}\\
& = -k\frac{\phi_r}{r}- \frac{u^2}{r^2}.\numberthis\label{qeq}
\end{align*}
Upon integrating \eqref{mainsyspoisson} from $0$ to $r$,
\begin{align}
\label{phirlim}
-r^{N-1}\phi_r + \lim_{r\to 0^+} r^{N-1}\phi_r & = \int_0^r(\rho-c)y^{N-1}\, dy.
\end{align}
%Assuming the appropriate boundary condition on $\phi$ that the above limit is zero and taking time derivative, 
Local well-posedness requires specifying a boundary condition at the origin. The standard convention is to impose a zero boundary conditions, meaning that the corresponding limit vanishes as $r\to 0$. Likewise  the density flux ($\rho ur^{N-1}$) tends to zero, indicating that no mass is lost at the origin. 

Taking the time derivative of \eqref{phirlim}, we obtain 
\begin{align*}
-r^{N-1}\phi_{rt} & = \int_0^r \rho_t y^{N-1}\, dy\\
& = -\int_0^r (\rho u y^{N-1})_y\, dy\qquad \text{from \eqref{mainsysmass}}\\
& = -\rho ur^{N-1} + \lim_{r\to 0^+} \rho ur^{N-1}\\
& = -\rho ur^{N-1}.
\end{align*}
%Assuming boundary condition on $\rho,u$ that the limit term to be zero, 
As a consequence, we obtain $\phi_{rt} = \rho u$. We now use this identity, together with the previously derived  expression for $\phi_{rr}$,  in the following calculation. 
\begin{align*}
\left(\frac{\phi_r}{r}\right)_t + u\left(\frac{\phi_r}{r}\right)_r & = \frac{\phi_{rt}}{r} + u\frac{\phi_{rr}}{r} - \frac{u}{r}\frac{\phi_r}{r}\\
& =  \frac{\rho u}{r} - (N-1)\frac{u}{r}\frac{\phi_r}{r} - (\rho-c)\frac{u}{r} - \frac{u}{r}\frac{\phi_r}{r}\\
& = - N\frac{u}{r}\frac{\phi_r}{r} +c\frac{u}{r}.\numberthis\label{seq}
\end{align*}
Set 
\begin{align}
\label{vartransf}
p:=u_r,\quad q:=\frac{u}{r},\quad s:=-\frac{\phi_r}{r}.
\end{align} 
Using equations \eqref{rhoeq}, \eqref{peq}, \eqref{qeq} and \eqref{seq}, we obtain the following ODE system,
\begin{subequations}
\label{fullodesys}
\begin{align}
& \rho' = -(N-1)\rho q - p\rho, \label{rhoode}\\
& p' = -p^2 - k(N-1)s + k(\rho -c), \label{pode}\\
& q' = ks-q^2, \label{qode}\\
& s' = -q(c+Ns), \label{sode}
\end{align}
\end{subequations}
with initial data 
$$
\rho_0 := \rho(0,\beta), p_0:=p(0,\beta), q_0:=q(0,\beta), s_0:=s(0,\beta).
$$ 
Here, prime $'$ denotes differentiation along the characteristic path \eqref{chpath}. Note that we continue to use the notation $\rho$ for the density as in \eqref{mainsysmass}, which is a function of time and the spatial variable, as well as for the solution to the ODE \eqref{rhoode}, which is another function of time only (for a fixed parameter $\beta$). Similarly for the notation $\rho_0$. However, it will be clear from context whether we are referring to $\rho$ as the solution of the main PDE system \eqref{mainsys} or as the unknown in \eqref{rhoode}. 

For the special case $N=4$, we introduce the auxiliary variable (which follows the same notation overload as noted above),
\begin{align*}
    & f_4 :=  \frac{c-\rho+4s}{4\rho\left(s+\frac{c}{4}\right)^{\frac{1}{4}}},
\end{align*}
which satisfies the second-order differential equation
\begin{align}
\label{eq:perassum}
f_4'' + k(c+3s)f_4 = 0.
\end{align}
 We will express the periodicity assumption in terms of this quantity. Later, we will show that for any $N\geq 2$, the solutions of the decoupled system \eqref{qode}-\eqref{sode} are periodic. 
 For the case $N=4$, we will make the following assumption.
\begin{assump}
\label{perassump}
The fundamental matrix of \eqref{eq:perassum} is periodic,  with the same period  as the system \eqref{qode}-\eqref{sode}.
\end{assump}
This assumption is motivated by numerical observations.

\section{The $q-s$ system}
\label{secqs}
Observe that \eqref{qode}, \eqref{sode} are decoupled from \eqref{rhoode}, \eqref{pode}. In fact, one can obtain a closed-form  expression for the trajectory curve associated with the system  \eqref{qode}-\eqref{sode}. The following  Lemma describes the behavior of the solutions. 
\begin{lemma}
\label{qsregularity}
(Global behavior of the system \eqref{qode}-\eqref{sode}).
\begin{enumerate}
\item Existence and boundedness. \\
The functions $s(t),q(t)$ exist for all time and remain uniformly bounded if and only if 
$$
s_0>-\frac{c}{N}.
$$
In particular, if $s_0\leq -\frac{c}{N}$, then $q\to-\infty$ in finite time.
\item Trajectory curves. \\
When $s_0>-\frac{c}{N}$, the solutions lie on bounded trajectory curves given by 
\begin{align}
& \frac{1}{\left(s+\frac{c}{2}\right)}\left( q^2 + \frac{kc}{2} + k\left( s+\frac{c}{2}\right) \ln\left(s+\frac{c}{2}\right) \right) = \text{constant},\qquad N=2, \label{trajectory2}\\
& \left( s+\frac{c}{N}\right)^{-\frac{2}{N}}\left( q^2 + \frac{kc}{N} + \frac{2k\left(s+\frac{c}{N}\right)}{N-2} \right)=\text{constant},\qquad N>2. \label{trajectoryN}
\end{align}
\item Periodicity and rotation direction. \\
When $s_0>-\frac{c}{N}$, the solutions are periodic,  and their trajectories rotate clockwise on the $(q,s)$ plane as time increases.
\end{enumerate}
\end{lemma}
We will frequently use the shifted variable,
\begin{align}
\label{tildestrans}
\tilde s:= s+\frac{c}{N},\quad \tilde s_0:= s_0 + \frac{c}{N}, 
\end{align}
and the system \eqref{tildesode} in place of $s$ and \eqref{sode} as the situation demands,  since this substitution simplifies subsequent calculations. The equivalent $q,\tilde s$ system is,
\begin{subequations}
\label{qtildessys}
\begin{align}
& q' = k\tilde s - \frac{kc}{N} -q^2,\label{qtildesode} \\
& \tilde s' = -qN\tilde s. \label{tildesode}
\end{align}
\end{subequations}
\textit{Proof of Lemma \ref{qsregularity}:}
It is natural to use $\tilde s$ rather then $s$. We begin with the ``only if" part of the first assertion, proved via contradiction. 

From \eqref{tildesode},  $\tilde s(t)$ maintains its sign as long as $q(t)$ exists. Suppose $\tilde s_0\leq 0$, then $\tilde s(t) \leq 0$ for as long as the solution exists. Assume, for contradiction, that $q(t)$ remains bounded for all $t$.  From \eqref{qtildesode} we have,
$$
q' \leq -\frac{kc}{N} -q^2 \leq \min\left\{ -\frac{kc}{N}, -q^2\right\}.
$$
Since the right hand is strictly negative, there exists a finite time, $t_-\geq 0$ such that $q(t_-)<0$. Once $q$ is negative, its evolution satisfies a Riccati-type inequality, 
$$
q' <-q^2,\qquad q(t_-)<0,
$$
which implies that $q(t)$ blows up to $-\infty$ in finite forward time: 
$$
\lim_{t\to t_c^-} q = -\infty \quad \text{for some}\quad  t_c< t_- - 1/q(t_-).
$$
This contradicts the assumed global boundedness of $q$. Hence $\tilde s_0 >0$ is necessary.  

Next, we establish the ``if" part of the first assertion and derive the trajectory equations (the second assertion).  Assume $\tilde s_0 >0$.  Then $\tilde s(t)>0$ for all $t>0$ as long as the solution exists. Dividing \eqref{qtildesode} by \eqref{tildesode} yields 
\begin{align*}
& q\frac{dq}{d\tilde s}  = -\frac{k}{N} + \frac{kc}{N^2 \tilde s} + \frac{q^2}{N\tilde s},\\
& \frac{d(q^2)}{d\tilde s} - \frac{2q^2}{N\tilde s}  = \frac{2k}{N^2}\left(  \frac{c}{\tilde s} - N \right).
\end{align*}
Multiplying by the integrating factor $\tilde s^{-\frac{2}{N}}$, we obtain 
\begin{align*}
& \frac{d}{d\tilde s}\left( q^2\tilde s^{-\frac{2}{N}} \right) = \frac{2k}{N^2}\left( c\tilde s^{-1-\frac{2}{N}} - N\tilde s^{-\frac{2}{N}} \right).
\end{align*}
Because of the different exponents, the cases $N>2$ and $N=2$ must be treated separately. 
%Owing to the last term, we see that $N=2$ case has to be handled separately before integrating the above equation. 
%\begin{align*}
%& q\left(\tilde s^{-1/N} \frac{dq}{d\tilde s} - \frac{q\tilde s^{-1-1/N}}{N}\right) + \frac{k\tilde s^{-1/N}}{N}\left( 1- \frac{c}{N\tilde s} \right) = 0, \\
%& q\frac{d\left(q\tilde s^{-1/N}\right)}{d\tilde s} + \frac{k\tilde s^{-1/N}}{N}\left( 1- \frac{c}{N\tilde s} \right) = 0, \\
%& \left( q\tilde s^{-1/N} \right)\frac{d\left(q\tilde s^{-1/N}\right)}{d\tilde s} + \frac{k}{N}\left( \tilde s^{-2/N} - \frac{c\tilde s^{-1-2/N}}{N} \right) = 0.
%\end{align*}
Case $N>2$. \\
Integrating gives 
\begin{align*}
%& \left(q\tilde s^{-1/N}\right)^2 + \frac{2k\tilde s^{1-2/N}}{N-2} + \frac{kc\tilde s^{-2/N}}{N} = R_0,
R_N(q,\tilde s) :=\tilde s^{-2/N}\left(q^2 + \frac{2k\tilde s}{N-2} + \frac{kc}{N}\right) = \text{constant}.
\end{align*}
%where $R_N:= \tilde s_0^{-2/N}\left( q_0^2 + \frac{2k\tilde s_0}{N-2} + \frac{kc}{N} \right)>0$. 
Case $N=2$. \\
Integration yields 
\begin{align*}
& R_2(q,\tilde s):=\frac{q^2}{\tilde s} + k\ln(\tilde s) + \frac{kc}{2\tilde s} = \text{constant}.
\end{align*}
%where $R_2:= \frac{q_0^2}{\tilde s_0} + k\ln(\tilde s_0) + \frac{kc}{2\tilde s_0}>0$.
Finally, from the trajectory relation for $N>2$, we obtain the bound ( valid whenever $\tilde s_0>0$),
\begin{align*}
& \tilde s < \left(\frac{R_N(N-2)}{2k}\right)^{\frac{N}{N-2}},\\
& |q|< \tilde s^{1/N}\sqrt{R_N}< R_N^{\frac{N}{2N-4}} \left(\frac{N-2}{2k}\right)^{\frac{N}{N-2}} .
\end{align*}
Hence, the solutions are uniformly bounded. Analogous bounds for $q,\tilde s$ can be derived in the case $N=2$ as well. 
This completes the proof to the first and second assertions. 

Observe that the trajectory equations are invariant under the transformation $q\to -q$. Combined with the fact that linearized system around the unique critical point, $(0,c/N)$ has purely imaginary eigenvalues, we conclude that every trajectory starting at any $\tilde s_0>0$ forms a closed curve encircling the critical point. Hence, the solutions are periodic. 

The direction of motion along  the trajectory follows directly from \eqref{tildesode}:
$$
\tilde s'<0 \quad \text{when} \quad  q>0 , \quad  \tilde s'>0 \quad \text{when}  \quad  q<0, 
$$
which shows that the trajectories move clockwise in the $q-\tilde s$-plane as time increases.

\qed

\begin{figure}[htb]
	\centering
    \input{TikzCode/Fig-qs.tex}
    
    \caption{Illustration of Theorem \ref{ctcn} with $N = 5, k=1, c = 1, q_0 = 0.2, s_0 = -0.1$.}
	\label{figqsphase}
\end{figure}

\begin{corollary}
\label{qsalltime}
The functions $q(t),s(t)$ in the  system \eqref{qode}-\eqref{sode} exist for all time. In particular, Assertions $2$ and $3$ of Lemma \ref{qsregularity} apply.
\end{corollary}
\begin{proof}
Using the expression of $s$ from \eqref{vartransf} in \eqref{phirlim} at $t=0$, we have,
\begin{align*}
    s_0 & = -\frac{\phi_{0r}(\beta)}{\beta} = \frac{1}{\beta^N}\int_0^\beta (\rho_0(\xi)-c) \xi^{N-1}\, d\xi \\
    & > -\frac{c}{N}.
\end{align*}
Therefore, by Lemma \ref{qsregularity}, the conclusion follows. 
\end{proof}

\begin{remark}
\label{rempoissonforce}
This result asserts that in multiple dimensions, the Poisson forcing alone prevents  concentrations at the origin -- irrespective of how negative the initial velocity may be. This phenomenon does not hold in one dimension. A more detailed comparison between one-dimensional and multi-dimensional  cases is provided in Section \ref{secconc} (conclusion).
\end{remark}

We now set up notation for the trajectory curve obtained in Lemma \ref{qsregularity}, which will be used  here  and  in later sections. Define 
\begin{align}
\label{Rconstant}
R_N(q,\tilde s) = \left\{ 
\begin{array}{c}
    \tilde s^{-2/N}\left( q^2 + \frac{2k\tilde s}{N-2} + \frac{kc}{N} \right),\quad N>2,   \\
    \frac{q^2}{\tilde s} + k\ln(\tilde s) + \frac{kc}{2\tilde s},\quad N=2.
\end{array}
\right.
\end{align}
From Lemma \ref{qsregularity},
every solution satisfies 
$$
R_N(q,\tilde s) = R_N(q_0,\tilde s_0).
$$
When the value is fixed, we simply write $R_N$ to denote this constant; when needed as a function of $(q,\tilde s)$. 

The next Lemma pertaining to \eqref{qtildessys} will be useful in proving the blowup and global existence results in Section \ref{secglobal}.
\begin{lemma}
\label{qslemapp}
Consider the system \eqref{qode}-\eqref{sode}. Denote the two coordinates at which the trajectory curve intersects the $s$-axis by $(0,s_{min})$ and $(0,s_{max})$. Then
    $$
    0 <-s_{min}< s_{max}.
    $$
Moreover, the algebraic equation
$$
R_N(0, \tilde s) = R_N(q_0,\tilde s_0),
$$
%$$
%\frac{2k}{N-2} y^{1-\frac{2}{N}} + \frac{kc}{N}y^{-\frac{2}{N}} - R_N = 0,\ y>0, \quad N>2,
%$$
%$$
%k\ln(y) + \frac{kc}{2}y^{-1} - R_2 = 0,\ y>0,
%$$
(for positive arguments) has exactly two roots, which are $s_{min}+\frac{c}{N}$ and $s_{max}+\frac{c}{N}$.
%$R_0,R_{2_0}$ are the constant right-hand-side in \eqref{trajectoryN} and \eqref{trajectory2} respectively.
\end{lemma}
\begin{proof}
We give a proof for $N>2$ only. Very similar arguments apply for $N=2$.  Note that for $\tilde s$, the statement is equivalent to saying, 
\begin{align}
\label{sminmaxineq}
& 0<\frac{c}{N} - \tilde s_{min}< \tilde s_{max}-\frac{c}{N} , 
\end{align}
where $\tilde s_{max}:=s_{max}+c/N$, $\tilde s_{min}:= s_{min}+c/N$.

Since the solution trajectories, $(q,\tilde s)$, are bounded periodic orbits around $(0,c/N)$, $(0,\tilde s_{min})$ and $(0,\tilde s_{max})$ lie on either side of $(0,c/N)$ implying that $c/N - \tilde s_{min}>0$.

Now consider the function, 
$$
g(\tilde s):= \frac{2k}{N-2}\tilde s^{1-\frac{2}{N}} + \frac{kc}{N}\tilde s^{-\frac{2}{N}}.
$$
This is essentially the left-hand-side of equation \eqref{trajectoryN} with $q=0$. We aim to show that $g(\tilde s)=R_N$ has exactly two positive roots. 

Observe that $g$ goes to infinity as $\tilde s\to 0$ or $\tilde s\to\infty$. Moreover, 
$$
\frac{dg}{d\tilde s} = \frac{2k}{N}\tilde s^{-\frac{2}{N}} - \frac{2kc}{N^2}\tilde s^{-1-\frac{2}{N}}. 
$$
Setting the derivative as zero yields a unique critical point at $\tilde s = \frac{c}{N}$, which is a global minimum. Hence, 
$$
g:(0,\infty )\longrightarrow (g(c/N),\infty)
$$
is strictly decreasing on $(0, c/N)$ and strictly increasing on $(0, c/N)$. 

Given the structure of $g$, we have that the algebraic equation,
$$
g(\tilde s) = R_N, 
$$
has exactly two roots if the constant, $R_N> g\left(\frac{c}{N}\right)$. These are precisely  $\tilde s_{min}$ and $\tilde s_{max}$. 

Clearly, if $\tilde s_{max}\geq \frac{2c}{N}$, the second inequality in \eqref{sminmaxineq} holds automatically since non-positive real numbers are not in the domain of $g$. In other words, we only need to prove the inequality for 
$$
R_N< g\left(\frac{2c}{N}\right)
$$
or equivalently, for  $\tilde s_{max}<2c/N$. We achieve this by making use of the sign of the third derivative of $g$:
$$
\frac{d^3g}{d\tilde s^3} = \frac{4k(N+2)}{N^4}\left( \tilde s - \frac{c}{N}\left(N+1\right) \right) \tilde s^{-3-\frac{2}{N}}.
$$
Since,
$$
\frac{c}{N}(N+1) > \frac{2c}{N}, 
$$
it follows that for all $\tilde s\in (0,2c/N)$, $\frac{d^3g}{d\tilde s^3} < 0$. 
Consequently, the second derivative is strictly decreasing on this interval, and therefore 
\begin{align*}
& \frac{d^2g}{d\tilde s^2}(\tilde s_1) > \frac{d^2g}{d\tilde s^2}(c/N) > \frac{d^2g}{d\tilde s^2} (\tilde s_2),\quad \tilde s_1\in \left(0,\frac{c}{N}\right),\ \tilde s_2\in\left( \frac{c}{N},\frac{2c}{N}\right).
\end{align*}
Consequently, for any $\delta <c/N$,
\begin{align*}
& \int_{c/N-\delta}^{c/N} \frac{d^2g}{d\tilde s^2}(\tilde s_1)d\tilde s_1 > \delta \frac{d^2g}{d\tilde s^2}(c/N) > \int_{c/N}^{c/N+\delta} \frac{d^2g}{d\tilde s^2}(\tilde s_2) d\tilde s_2,
\end{align*} 
and hence 
\begin{align*}
 -\frac{dg}{d\tilde s} (c/N-\delta) > \frac{dg}{d\tilde s} (c/N+\delta).
\end{align*}
The second inequality is a result of the fact that $\frac{dg}{d\tilde s}(c/N)=0$. Integrating both sides  with respect to $\delta$ from zero  upward yields 
$$
g(c/N-\delta) > g(c/N+\delta).
$$
Consequently, for the same shift from $\tilde s=c/N$ the function attains a higher value on the left than on the right. Therefore, the two points, $\tilde s_{min},\tilde s_{max}$, in the level set 
$$
\{\tilde s: g(\tilde s)=R_N\}
$$
must satisfy  
$$
\frac{c}{N}- \tilde s_{min} < \tilde s_{max} - \frac{c}{N}.
$$
This completes the proof. 
%Similarly, for $N=2$, from \eqref{trajectory2} we have,
%\begin{align*}
%& \tilde s < e^{\frac{R_{2_0}}{k}},\\
%& |q|<\sqrt{ R_{2_0}\tilde s + \frac{k}{e}}< \sqrt{ R_{2_0} e^{\frac{R_{2_0}}{k}}+ \frac{k}{e}} .
%\end{align*}
\end{proof}
We also include a short lemma that establishes a different relationship between $\tilde s$ and $q$,  which will be useful in later sections. To this end, we first define the following quantity,
\begin{align}
\label{Gammaexp}
\Gamma(t) = e^{-\int_0^t q(\tau)\, d\tau}. 
\end{align}
\begin{lemma}
\label{tildesexpq}
We have 
\begin{align}
\label{expqex}
\Gamma(t) = \left(\frac{\tilde s (t)}{\tilde s_0}\right)^{\frac{1}{N}},
\end{align}
for all $t>0$. In particular, 
$$
\int_0^T q(t) dt = 0
$$ and $\Gamma(t)$ is uniformly bounded and periodic. Here, $T$ is the period of the system \eqref{qtildessys} (equivalently \eqref{qode},\eqref{sode}).
\end{lemma}
\begin{proof}
From Corollary \ref{qsalltime}, $\tilde s(t),q(t)$ exist for all time,   and $\tilde s$ remains strictly  positive. Since $\Gamma' = -q \Gamma$, we may divide this by equation \eqref{tildesode} to obtain, 
$$
\frac{d\Gamma}{d\tilde s} = \frac{\Gamma}{N\tilde s}.
$$
Integrating yields \eqref{expqex}. The uniform boundedness of $e^{-\int_0^t q}$ follows immediately from the uniform boundedness of $\tilde s$. Moreover, since $\tilde s$ is periodic with period $T$, we have that for any $t\in\mathbb{R}$,
$$
e^{-\int_0^t q} = e^{-\int_0^{t+T} q}.
$$
Therefore,
\begin{align*}
& \int_0^{T} q =0. 
%& = \int_t^0 q + \int_0^T q + \int_T^{t+T}q\\
%& = \int_t^0 q + \int_0^T q + \int_0^t q\\
%& = \int _0^T q.
\end{align*}
%To obtain the last equality, we used the fact that $q$ has period $T$. 
\end{proof}

\section{Blow up of solutions}
\label{secblowup}
We now turn to the analysis of equations \eqref{rhoode} and \eqref{pode}. Our aim in this section is to entail conditions under which $p,\rho$ may blow up. From Corollary \ref{qsalltime}, we already know that $q,s$ are uniformly bounded, periodic. Hence, the quantities in \eqref{fullodesys} that can blowup are $p$ or $\rho$.
%Therefore, from this point onwards, we will assume that $s(0)>-c/N$ or equivalently, $s(t)>-c/N$ for all $t>0$. From Lemma \ref{qsregularity}, $q,s$ are uniformly bounded and periodic.

If $\rho_0=0$, then from \eqref{rhoode}, $\rho\equiv 0$ for as long as $\rho,p$ exist. Substituting this into  \eqref{pode} yields, 
$$
p'\leq -p^2-kc,
$$
which leads to Riccati-type blow up of $p$. Consequently, zero initial density density always leads to blow-up. From this point onward we therefore assume $\rho_0>0$. This in turn implies $\rho(t)>0$ for $t>0$ for as long as $\rho,p$ exist. 

With this in place, we proceed to simplify the system \eqref{rhoode}-\eqref{pode}, following the approach in \cite{Tan21}. From \eqref{rhoode}-\eqref{pode}, we obtain
\begin{align*}
\left(\frac{1}{\rho}\right)' & = q(N-1)\frac{1}{\rho} + \frac{p}{\rho},\\
\left(\frac{p}{\rho}\right)' & = -k(c+s(N-1))\frac{1}{\rho} + q(N-1)\frac{p}{\rho} + k.
\end{align*}
Observe that the coefficients of $1/\rho,p/\rho$ in their respective ODEs are the same, we can multiply both ODEs by integrating factor, 
$$
e^{-(N-1)\int_0^t q},
$$
to obtain,
\begin{align*}
& \left(\frac{1}{\rho}e^{-(N-1)\int_0^t q}\right)' = \frac{p}{\rho} e^{-(N-1)\int_0^t q},\\
& \left(\frac{p}{\rho}e^{-(N-1)\int_0^t q}\right)' = -k(c+s(N-1))\frac{1}{\rho}e^{-(N-1)\int_0^t q} + ke^{-(N-1)\int_0^t q}.
\end{align*}
Setting 
\begin{align} 
\label{etawtransf}
\eta:= \frac{1}{\rho}\Gamma^{N-1},\qquad w = \frac{p}{\rho}\Gamma^{N-1}, 
\end{align}
we obtain the new system,
\begin{subequations}
\label{simplestODE}
\begin{align}
& \eta' = w, \label{simplestODEeta}\\
& w' = -k\eta(c+s(N-1)) + k\Gamma^{N-1}. \label{simplestODEw}
\end{align}
\end{subequations}
We label the corresponding initial data as $\eta_0,w_0$.
\begin{remark}
\label{remlinsys}
Owing to Corollary \ref{qsalltime} and Lemma \ref{tildesexpq}, all coefficients in the linear ODE system \eqref{simplestODE} are uniformly bounded. Hence, $\eta(t),w(t)$ remain bounded and well-defined for all $t\in(-\infty,\infty)$. Consequently, the key to global existence of the original variables is to ensure 
$$
\eta(t)>0 \quad  \text{for all}\quad  t>0.
$$
From \eqref{etawtransf}, this would imply that $p,\rho$ are both bounded for all $t>0$. Conversely, if there is a finite time, $t_c$, at which $\eta$ becomes zero, then
$$
\lim_{t\to t_c^-} \rho(t) = \lim_{t\to t_c^-} \frac{1}{\eta (t)} = \infty.
$$
Moreover, since $w$ is bounded for all times,
$$
\lim_{t\to t_c^-} |p(t)| =  \lim_{t\to t_c^-}\frac{|w(t)|}{\eta(t)} = \infty.
$$
Therefore, at the moment of breakdown, $\rho,|p|$ blow up 
simultaneously.
\end{remark}
The above remark results in the following key proposition.
\begin{proposition}
\label{rhoptogether}
Suppose $\rho_0\neq 0.$ Then $\rho,p,q,s$ in \eqref{fullodesys} are well-defined for all $t>0$ if and only if 
$$
\eta(t)>0 \quad \text{for all}\quad  t>0
$$
in the system \eqref{simplestODE}. In particular, if there is a $t_c>0$ such that $\eta(t_c)=0$, then 
$$
\lim_{t\to t_c^-} |p(t)| =  \lim_{t\to t_c^-}\rho(t) = \infty.
$$
\end{proposition}

\iffalse
We now state a result for the trivial case.
\begin{proposition}
\label{trivialprop}
Suppose $(q_0,s_0)=(0,0)$. Consider \eqref{simplestODE} with initial data $(\eta(0),w(0))$. Then $\eta(t)>0$ for all $t>0$ if and only if 
$$
kc\eta^2(0) + w^2(0) < 2k\eta(0).
$$
\end{proposition}
\begin{proof}
$(0,0)$ is the critical point of the system \eqref{qode},\eqref{sode}. Hence, $q\equiv 0\equiv s$. \eqref{simplestODE} reduces to,
$$
\eta'' + kc\eta = k.
$$
Therefore,
$$
\eta(t) = \frac{1}{c} + \left(\eta(0)-\frac{1}{c}\right) \cos(t\sqrt{kc}) + \frac{w(0)}{\sqrt{kc}}\sin(t\sqrt{kc}).
$$
Consequently, positivity of $\eta$ is guaranteed if and only if,
$$
\left( \eta(0)-\frac{1}{c}\right)^2 + \frac{w^2(0)}{kc} < \frac{1}{c^2},
$$
which completes the proof.
\end{proof}
From hereafter, we will assume that the solution $(s,q)$ is not the one that corresponds to the single point equilibrium solution. In particular, $(s(t),q(t))\neq (0,0)$ for all $t>0$.
\fi

Next, we introduce one of the key nonlinear quantities  of this paper,  which will be instrumental in analyzing the system \eqref{rhoode},\eqref{pode}. Define 
\begin{align}\label{Aexp}
A:= qw - k\eta s.
\end{align}
Using \eqref{qode},\eqref{sode} and the system \eqref{simplestODE},
\begin{align*}
A'& = q'w + qw' - k\eta' s - k\eta s'\\
& = (ks-q^2)w + q\left(-k\eta(c+s(N-1)) + k\Gamma^{N-1}\right) - kws + kq\eta ( c+Ns)\\
& = -q^2w + kq\eta s + kq\Gamma^{N-1}\\
& = -q(qw - k\eta s) + kq\Gamma^{N-1}\\
& = -qA + kq\Gamma^{N-1}. \numberthis \label{Aode}
\end{align*}
We have,
$$
A' + qA = -\left(\frac{k}{N-1}\Gamma^{N-1}\right)'.
$$
Throughout this subsection we assume $N> 2$. As evident in the calculations below, the $N=2$ case has to be handled separately and we will tackle it at the end of the section. Upon integration and using the notaiton 
$$
A_0 := q_0w_0 - k\eta_0s_0,
$$ 
we obtain,  
\begin{align*}
A\Gamma^{-1} -A_0 & = -\frac{k}{N-1}\int_0^t (\Gamma(s))^{-1} \left( \Gamma^{N-1}(s) \right)' ds\\
& = -\frac{k}{N-1}\left[\Gamma^{N-2} - 1 - \int_0^t q(s) (\Gamma(s))^{N-2} ds  \right].
\end{align*} 
Using $\Gamma'=-q\Gamma$, the integral simplifies further, giving  
\begin{align*} 
A\Gamma^{-1} -A_0 & = -\frac{k}{N-1}\left[\Gamma^{N-2} - 1 +\frac{1}{N-2}\left( \Gamma^{N-2} - 1 \right)  \right]\\
& = \frac{k}{N-2}\left[ 1 - \Gamma^{N-2} \right].
\end{align*}
Hence,
\begin{align}
\label{Aeq}
\begin{aligned}
& A(\Gamma) = \left( A_0 + \frac{k}{N-2} \right)\Gamma - \frac{k}{N-2}\Gamma^{N-1},\quad \Gamma(t) := e^{-\int_0^t q}>0,\\
& A(t) := A(\Gamma(t)). 
%A_0e^{-\int_0^t q} + \frac{k}{N-2}\left( e^{-\int_0^t q} - e^{-(N-1)\int_0^t q} \right),
\end{aligned}
\end{align}
Here, we have slightly abused the notation: 
%$A$ and $\Gamma$ by assigning both of them to two functions. 
$A(t)$ denotes the time-dependent function, while  $A(\Gamma)$ denotes the function of the scalar variable $\Gamma$ as in \eqref{Aeq}. Likewise, $\Gamma$ denotes  the argument of the function $A(\Gamma)$ as well as a function itself as in \eqref{Gammaexp}, that takes positive values. The context will make the intended meaning clear. 
%It will, however, be clear from context which functions we are referring to.
\begin{remark}
\label{remAexp}
Remarkably, $A(t)$ depends on $\eta,w$ only through the initial data, $\eta_0,w_0$. Thus the system consisting of \eqref{qode}, \eqref{sode}, \eqref{Aode} with initial data $q_0,s_0,A_0$ forms an initial value problem that evolves independent of the dynamics in \eqref{simplestODE}.
%Once the initial data to \eqref{fullodesys} is given, $A(t)$ is a known function depending only on the solution of the system, \eqref{qode} and \eqref{sode}, $(q,s)$. 
Moreover, $A(t)$ can be explicitly solved as above. By  Lemma \ref{tildesexpq}, $A(t)$ has the same period as $q,s$.
\end{remark}
We now  give a sufficient condition for the vanishing of $\eta$ in finite time.
\begin{proposition}
\label{AsignFTB}
If $A(t)$ is nonnegative (or nonpositive) for all $t$, then there exists $t_c>0$ such that $\eta(t_c) = 0$.
\end{proposition}
\begin{proof}
Suppose $A(t)\geq 0$ for all $t$. Let $t_0$ be a time at which $s$ attains its maximum. From  equaiton \eqref{sode} we have $q(t_0)=0$. Therefore  
\begin{align*}
& 0\leq A(t_0) = q(t_0)w(t_0) -k\eta(t_0)s(t_0)  = -k\eta(t_0)s(t_0).
\end{align*}
From Lemma \ref{qslemapp}, $s(t_0)>0$. Hence, $\eta(t_0)\leq 0$. Since $\eta$ is continuous and
$\eta(0)=\eta_0>0$ (by assumption), there must exist a time
$$
0<t_c \leq t_0
$$
such that $\eta(t_c)=0$. 

The proof is similar for the case $A(t)\leq 0$ for all $t$. 
\end{proof}

\begin{remark}
\label{remetaendpoints}
At the times where $s$ achieves its extrema, the value of $\eta$ can be determined a priori;  it depends only on $A,q$ and $s$. In fact, a simple computation gives $\eta$ explicitly at these times. Indeed, if $s$ achieves a maximum or minimum  at $t=t_\ast$, then $q(t_\ast)=0$. Hence, by \eqref{Aexp},
$$
\eta(t_\ast) = -\frac{A(t_\ast)}{ks(t_\ast)}.
$$
Since $A,s$ are periodic with the same period $T$, it follows that 
$$
\eta(t_\ast) = \eta(t_\ast+lT), l=0,1,2,\ldots.
$$
These coordinates are marked in Figure \ref{figlemeq}.
\end{remark}

%Using \eqref{Aeq} in \eqref{simplestODEw}, we obtain another ODE for $w$,
%\begin{align}
%\label{simplestODEw2}
%& w' = -kc\eta - q (N-1)w + A(N-1)+ ke^{-(N-1)\int_0^t q}.
%\end{align}
%Let $q\in[-q_M,q_M]$.
\begin{lemma}
\label{Asign}
Suppose $N\geq 3$. Then:
\begin{itemize}
    \item If $1+\frac{A_0(N-2)}{k}\leq 0$, then $A(t)<0$ for all $t>0$.
    \item If 
    %$0>A_0>-\frac{k}{N-2}$ and 
    $\int_0^t q \leq -\frac{1}{N-2}\ln\left( 1+\frac{A_0(N-2)}{k} \right)$ for all $t>0$, then $A(t)\leq 0$ for all $t>0$.
    \item If 
    %$A_0\geq 0$ and 
    $\int_0^t q \geq -\frac{1}{N-2}\ln\left( 1+\frac{A_0(N-2)}{k} \right)$ for all $t>0$, then $A(t)\geq 0$ for all $t>0$.
\end{itemize}
\end{lemma}
\begin{proof}
The first assertion follows from the fact that if $1+\frac{A_0(N-2)}{k}\leq 0$, then from \eqref{Aeq}, $A(t)<0$ for all time.

%\begin{align}
%\label{Agamma}
%& A(\Gamma) = \left(A_0 +\frac{k}{N-2}\right) \Gamma - \frac{k}{N-2}\Gamma^{N-1},\quad \Gamma = e^{-\int_0^t q}>0.
%\end{align}
Note that if $1+\frac{A_0(N-2)}{k}>0$, then it can be readily seen from \eqref{Aeq} that $A(\Gamma)$ has exactly two real roots, $0$ and 
\begin{align}
\label{kappaexp}
\kappa:= \left( 1+ \frac{A_0(N-2)}{k}\right)^{\frac{1}{N-2}}>0.
\end{align} 
Moreover, for nonnegative arguments $\Gamma$ in \eqref{Aeq}, $A(\Gamma)> 0$ for $\Gamma\in (0,\kappa)$ and $A(\Gamma)<0$ for $\Gamma>\kappa$. 

Now suppose the hypothesis of the second assertion holds. Straightforward calculations then imply $\Gamma(t)\geq \kappa$ for all $t$. Therefore, $A(\Gamma)\leq 0$ for all attainable values of $\Gamma$. This proves the second assertion. 

The third assertion is similar, only that here the hypothesis implies $\Gamma(t) \leq \kappa$. Therefore, $A(t)\geq 0$.
\end{proof}

\begin{corollary}
\label{ftbresult1}
Suppose $N\geq 3$. If one of the following is true,
\begin{itemize}
    \item $1+\frac{A_0(N-2)}{k}\leq 0$,
    \item $1+\frac{A_0(N-2)}{k}> 0$, and the two roots, $y_2>y_1>0$, of the equation,
    \begin{align}
    \label{tildesboundeq}
    \frac{2ky}{N-2} + \frac{kc}{N} - R_Ny^{2/N} = 0,\qquad (\text{equivalently } R_N(0,y) = R_N),
    \end{align}
    are such that,
    $$
    \kappa\notin \left(\left( \frac{y_1}{\tilde s_0} \right)^{1/N},\left(\frac{y_2}{\tilde s_0}\right)^{1/N}\right),
    $$
\end{itemize}
then there is a time $t_c>0$ such that $\eta(t_c)=0$. Here, $R_N$ is the constant in \eqref{Rconstant} and $\kappa$ is as in \eqref{kappaexp}.
\end{corollary}
\begin{remark}
All conditions in the hypothesis depend only on the initial data. Consequently,  they provide a characterization of the supercritical region for \eqref{mainsys}.
\end{remark}
\begin{proof}
If the first condition holds, then by the first assertion of Lemma \ref{Asign} we have $A(t)<0$ for all $t>0$. Proposition \ref{AsignFTB} then leads to the conclusion.

As noted earlier, equation  \eqref{tildesboundeq} is equivalent to $$
R_N(0,y) = R_N(q_0,\tilde s_0).
$$
Hence, by Lemma \ref{qslemapp}, the two roots of \eqref{tildesboundeq} correspond exactly to the maximum and minimum attainable values of $\tilde s$.
%Observe that the given algebraic equation in $y$ is in fact \eqref{trajectoryN} with $q=0,\tilde s=y$. Using \eqref{tildesode}, we see that the two roots of the given equation are the maximum and minimum attainable values of $\tilde s$ because $q=0$ when $\tilde s$ is max/min. The fact that the given equation will exactly have two roots follows from Lemma \ref{qslemapp}. 
Now suppose the hypothesis of the second assertion holds. There are two possible situations:
\begin{enumerate} 
\item  $\kappa\leq (y_1/\tilde s_0)^{1/N}\leq (\tilde s(t)/\tilde s_0)^{1/N}$ or,
\item $\kappa\geq (y_2/\tilde s_0)^{1/N}\geq (\tilde s(t)/\tilde s_0)^{1/N}$,
\end{enumerate}
for all $t>0$. Suppose $(1)$ holds.
Using Lemma \ref{tildesexpq},
$$
e^{-\int_0^t q} = \left(\frac{\tilde s(t)}{\tilde s_0}\right)^{\frac{1}{N}} \geq \kappa.
%\tilde s_0 \Gamma^N \leq  \tilde s_0\left( 1+\frac{A_0(N-2)}{k}\right)^{\frac{N}{N-2}}.
$$
Hence, $\int_0^t q \leq -\ln(\kappa)$ for all $t>0$. Using the second assertion of Lemma \ref{Asign}, we obtain that $A(t)\leq 0$ for all time. Then by Proposition \ref{AsignFTB}, we conclude the result. 

Case (2) is handled in exactly the same way, with the inequalities reversed. 
%Very similar arguments hold for (2) as well.
\end{proof}

We conclude this section by treating the case $N=2$. Integrating \eqref{Aode} with $N=2$ yields
\begin{align}
\label{A2eq}
\begin{aligned}
& B(\Gamma) :=  (A_0 - k\ln\Gamma )\Gamma,\\
& B(t) := B(\Gamma(t)),\quad \Gamma(t) = e^{-\int_0^t q}.
\end{aligned}
\end{align}
We continue to denote the initial value of $B$ by the same symbol $A_0$ for consistency. Proposition \ref{AsignFTB} remains valid with $B$ in place of $A$. Analogous to Lemma \ref{Asign}, we obtain: 
\begin{lemma}
\label{A2sign}
Suppose $N=2$. Then: 
\begin{itemize}
    \item If 
    $\int_0^t q \leq -\frac{A_0}{k}$ for all $t>0$, then $B(t)\leq 0$ for all $t>0$.
    \item If 
    $\int_0^t q \geq -\frac{A_0}{k}$ for all $t>0$, then $B(t)\geq 0$ for all $t>0$.
\end{itemize}
\end{lemma}
The proof is very similar to that of Lemma \ref{Asign} and fairly straightforward given the simplicity of \eqref{A2eq}.
\begin{remark}
Interestingly, the result of Lemma \ref{A2sign} can be related to that of Lemma \ref{Asign} by replacing,
$$
-\frac{1}{N-2}\ln\left( 1+\frac{A_0(N-2)}{k}\right),
$$
with
$$
-\lim_{\beta\to 2^+}\frac{1}{\beta-2}\ln\left( 1+\frac{A_0(\beta-2)}{k}\right).
$$
This shows that $N=2$ is indeed a critical case.
\end{remark}
Finally, we have the result analogous to Corollary \ref{ftbresult1}.
\begin{corollary}
\label{ftbresult2}
Suppose $N=2$. Let $y_2>y_1>0$ be the two roots  of 
    \begin{align}
    \label{tildesboundeq2}
    ky\ln (y) + \frac{kc}{2} - R_2y = 0,
    \end{align}
    where $R_2$ is the constant appearing on the right-hand side of \eqref{Rconstant}. If 
    $$
    e^{\frac{A_0}{k}}\notin \left(\sqrt{\frac{y_1}{\tilde s_0}},\sqrt{\frac{y_2}{\tilde s_0}}\right),
    $$
then there is a finite time $t_c>0$ such that $\eta(t_c)=0$. 
\end{corollary}
As before, the proof is very similar to the arguments used above and will therefore be omitted.
%Once again, the proof is very similar. 

Before we begin the new section, we  introduce some notation. We use $$
t_q,t_s,t_A
$$
to denote times at which  $q,s,A$ vanish,  respectively. Since these functions are periodic, there are infinitely many such times. When we need to distinguish among multiple instances, we write 
%Owing to their periodicity, there are infinitely many such times. In view of this, we will also use the notations 
$$
t_q^i,t_s^i,t_A^i,  \text{for} \quad  i\in\mathbb{N}
$$
to label the $i$-th occurrence of such times.
%to refer to more than one such times as and when it is required.
\section{Global Solution}
\label{secglobal}
Up to this point, we have narrowed down the class of initial data for \eqref{fullodesys} that may lead to $\eta(t)>0$, or equivalently, to the boundedness of $\rho(t)$ for all $t>0$. The goal of this section is to identify  precisely which initial configurations -- among those not already excluded in Section \ref{secblowup} -- ensure that  $\eta$ remains positive over one full period of the $q-s$ system. 

Therefore, from now on, we assume the negation of the hypothesis in Corollary \ref{ftbresult1}. This allows us to focus on the remaining, unresolved class of initial data.

For $N>2$, we assume 
\begin{align}
\label{gsseccond}
1+\frac{A_0(N-2)}{k}> 0, \qquad\kappa\in \left(\left( \frac{\tilde s_{min}}{\tilde s_0} \right)^{1/N},\left(\frac{\tilde s_{max}}{\tilde s_0}\right)^{1/N}\right),
\end{align}
with $\kappa$ as given in \eqref{kappaexp},  and $\tilde s_{min},\tilde s_{max}$ are the minimum and maximum attainable values of $\tilde s(t)$, which are also the two roots of \eqref{tildesboundeq}. 

For the case $N=2$, we set $\kappa= e^{\frac{A_0}{k}}$,  and $\tilde s_{min},\tilde s_{max}$ are the corresponding minimum and maximum attainable values of $\tilde s(t)$, given by the roots of \eqref{tildesboundeq2}. The first assumption in \eqref{gsseccond} becomes a vacuous statement when $N=2$. Under this notational convention,  the analysis and results presented in this section for $N>2$  is the same as those for $N=2$ with $A$ replaced by $B$. Hence,  we do not distinguish between the cases $N>2$ and $N=2$ throughout this section. At the end of this section, we will prove the all- time periodicity of $\eta$ in the 4D case, assuming a suitable periodicity condition. 

Owing to the firs assumption in \eqref{gsseccond}, $A(\Gamma)$ has a unique positive root, which is precisely $\kappa$. 
Moreover, from \eqref{expqex}, the second assumption in \eqref{gsseccond} implies that
$$
\kappa\in\{\Gamma(t): t>0\}^\circ,
$$
where $^\circ$ denotes the interior of the set. Since $\kappa$ is a root of $A(\Gamma)$, it follows from \eqref{Aeq} that $A(t)$ must change sign. 
Note that this is essentially the negation of the hypothesis of Proposition \ref{AsignFTB}; therefore, $A$ must necessarily change sign to hope for a situation wherein $\eta(t)>0$ for all $t>0$. See Figure \ref{figAgamma} for a visualization of this phenomenon.
\iffalse
Not only this, one of the following cases has to hold:
\begin{itemize}
    \item Suppose $A_0\geq 0$:\\
    Then from \eqref{Aroot}, we know $\kappa\geq 1$. Also, $t=0$ corresponds to $\Gamma=1$. So, we must have that 
    \begin{align}\label{gsAcond}
    e^{-\int_0^{t_M}q}>\kappa,
    \end{align}
    where $t_M$ is any time such that $s(t_M)=s_{max}>0$.
    \item Suppose $0> A_0>-\frac{k}{N-2}$:\\
    Then from \eqref{Aroot}, we know $\kappa < 1$. Once again, $t=0$ corresponds to $\Gamma=1$. So, we must have that 
\end{itemize}

We reiterate that the key is to ensure $\eta(t)>0$ for all $t>0$. From here onwards until specified, we consider $A_0\geq 0$. Hence $\kappa\geq 1$.
\fi

\begin{figure}[h!] 
\centering
\begin{tikzpicture}
	
	\definecolor{darkgrey176}{RGB}{176,176,176}
	\definecolor{darkviolet1910191}{RGB}{191,0,191}
	\definecolor{gold}{RGB}{255,215,0}
	\definecolor{lightgrey204}{RGB}{204,204,204}
	\definecolor{limegreen}{RGB}{50,205,50}
	
	\begin{axis}[
		legend cell align={left},
		legend style={
			font=\tiny,
			fill opacity=0.8,
			draw opacity=1,
			text opacity=1,
			at={(0.03,0.03)},
			anchor=south west,
			draw=lightgrey204
		},
		axis line style={line width=0.1pt},
		tick align=outside,
		tick pos=left,
		x grid style={darkgrey176},
		xlabel={\(\displaystyle \Gamma\)},
		xmin=-0.0733989401765789, xmax=1.54137774370816,
		xtick style={color=black},
		ylabel style={rotate=-90.0},
		ylabel={\(\displaystyle A\)},
		ytick=\empty,
		y label style={at={(axis description cs:0.1,0.7)}, anchor=west},
		ymin=-0.4007344675, ymax=0.4073474775
		]
		\addplot [line width=1pt, black]
		table {%
			0 0
			0.01 0.00633333
			0.02 0.0126666133333333
			0.03 0.01899973
			0.04 0.02533248
			0.05 0.0316645833333333
			0.06 0.03799568
			0.07 0.04432533
			0.08 0.0506530133333333
			0.09 0.05697813
			0.1 0.0633
			0.11 0.0696178633333333
			0.12 0.07593088
			0.13 0.08223813
			0.14 0.0885386133333333
			0.15 0.09483125
			0.16 0.10111488
			0.17 0.107388263333333
			0.18 0.11365008
			0.19 0.11989893
			0.2 0.126133333333333
			0.21 0.13235173
			0.22 0.13855248
			0.23 0.144733863333333
			0.24 0.15089408
			0.25 0.15703125
			0.26 0.163143413333333
			0.27 0.16922853
			0.28 0.17528448
			0.29 0.181309063333333
			0.3 0.1873
			0.31 0.19325493
			0.32 0.199171413333333
			0.33 0.20504693
			0.34 0.21087888
			0.35 0.216664583333333
			0.36 0.22240128
			0.37 0.22808613
			0.38 0.233716213333333
			0.39 0.23928853
			0.4 0.2448
			0.41 0.250247463333333
			0.42 0.25562768
			0.43 0.26093733
			0.44 0.266173013333333
			0.45 0.27133125
			0.46 0.27640848
			0.47 0.281401063333333
			0.48 0.28630528
			0.49 0.29111733
			0.5 0.295833333333333
			0.51 0.30044933
			0.52 0.30496128
			0.53 0.309365063333333
			0.54 0.31365648
			0.55 0.31783125
			0.56 0.321885013333333
			0.57 0.32581333
			0.58 0.32961168
			0.59 0.333275463333333
			0.6 0.3368
			0.61 0.34018053
			0.62 0.343412213333333
			0.63 0.34649013
			0.64 0.34940928
			0.65 0.352164583333333
			0.66 0.35475088
			0.67 0.35716293
			0.68 0.359395413333333
			0.69 0.36144293
			0.7 0.3633
			0.71 0.364961063333333
			0.72 0.36642048
			0.73 0.36767253
			0.74 0.368711413333333
			0.75 0.36953125
			0.76 0.37012608
			0.77 0.370489863333333
			0.78 0.37061648
			0.79 0.37049973
			0.8 0.370133333333333
			0.81 0.36951093
			0.82 0.36862608
			0.83 0.367472263333333
			0.84 0.36604288
			0.85 0.36433125
			0.86 0.362330613333333
			0.87 0.36003413
			0.88 0.35743488
			0.89 0.354525863333333
			0.9 0.3513
			0.91 0.34775013
			0.92 0.343869013333333
			0.93 0.33964933
			0.94 0.33508368
			0.95 0.330164583333333
			0.96 0.32488448
			0.97 0.31923573
			0.98 0.313210613333333
			0.99 0.30680133
			1 0.3
			1.01 0.292798663333333
			1.02 0.28518928
			1.03 0.27716373
			1.04 0.268713813333333
			1.05 0.25983125
			1.06 0.25050768
			1.07 0.240734663333333
			1.08 0.23050368
			1.09 0.21980613
			1.1 0.208633333333333
			1.11 0.19697653
			1.12 0.18482688
			1.13 0.172175463333333
			1.14 0.15901328
			1.15 0.14533125
			1.16 0.131120213333333
			1.17 0.11637093
			1.18 0.10107408
			1.19 0.0852202633333333
			1.2 0.0688
			1.21 0.05180373
			1.22 0.0342218133333334
			1.23 0.0160445300000001
			1.24 -0.00273791999999984
			1.25 -0.0221354166666666
			1.26 -0.0421579200000001
			1.27 -0.0628154700000001
			1.28 -0.0841181866666667
			1.29 -0.10607627
			1.3 -0.1287
			1.31 -0.151999736666667
			1.32 -0.17598592
			1.33 -0.20066907
			1.34 -0.226059786666667
			1.35 -0.25216875
			1.36 -0.27900672
			1.37 -0.306584536666667
			1.38 -0.33491312
			1.39 -0.36400347
		};
		\addlegendentry{\tiny{ The polynomial in \eqref{Aeq} } }
		\addplot [ultra thick, limegreen]
		table {%
			0.87650909935689 0
			1.46797880353158 0
		};
		\addlegendentry{\tiny{ Range of \(\displaystyle \Gamma\)} }
		\addplot [semithick, gold, mark=asterisk, mark size=3, mark options={solid}, only marks]
		table {%
			1.23856232963017 0
		};
		\addlegendentry{\tiny{ \(\displaystyle \kappa\)} }
		\addplot [semithick, black, mark=*, mark size=2.1, mark options={solid}, only marks]
		table {%
			1 0
		};
		\addlegendentry{\tiny{ \(\displaystyle \Gamma=1 (t=0)\)} }
		\addplot [semithick, blue, mark=square*, mark size=2.1, mark options={solid}, only marks]
		table {%
			0.87650909935689 0
		};
		\addlegendentry{\tiny{ corresponding to \(\displaystyle s_{min}\)} }
		\addplot [semithick, darkviolet1910191, mark=diamond*, mark size=2.5, mark options={solid}, only marks]
		table {%
			1.46797880353158 0
		};
		\addlegendentry{\tiny{ corresponding to \(\displaystyle s_{max}\)} }
		\addplot [line width=0.24pt, black, forget plot]
		table {%
			-0.0733989401765789 0
			1.54137774370816 0
		};
		\addplot [line width=0.24pt, black, forget plot]
		table {%
			0 -0.4007344675
			0 0.4073474775
		};
	\end{axis}
	
\end{tikzpicture}
\caption{$A$ vs $\Gamma$ with $N=5, k=1, c=1, q_0 = 0.2, s_0 = -0.1$. }
\label{figAgamma}
\end{figure}

From the expression for $\Gamma(t)$ in \eqref{Aeq},  the leftmost point on the green line in Figure \ref{figAgamma} corresponds to the time when $\tilde s$ attains its minimum and the rightmost point is when it attains the maximum. Combining this observation with \eqref{expqex} and \eqref{Aeq}, we obtain an important conclusion that will be used throughout this section. For any $t>0$,
\begin{align}
\label{Aextremesign}
\begin{aligned}
& \tilde s(t) = \tilde s_{max} \implies A(t)<0,\\
& \tilde s(t) = \tilde s_{min} \implies A(t)>0.
\end{aligned}
\end{align}
In much of this section, we will be analyzing the system \eqref{simplestODE}. Our goal is to obtain a necessary and sufficient condition to ensure $\eta(t)>0$ for all $t>0$. Recall from Remark \ref{remetaendpoints} that at the times when $s$ (or equivalently $\tilde s$) attains extrema values, the corresponding values of $\eta$ are known a priori. Specifically, if $t_q$ is such that $q(t_q)=0$, then from \eqref{Aexp}, 
\begin{align}
\label{etaatq0}
\eta(t_q) = -\frac{A(t_q)}{ks(t_q)}.
\end{align}
In particular, when $s(t_q) = s_{min}$,  we have 
$$
\eta(t_q) = -\frac{A(t_q)}{ks_{min}}.
$$
From Lemma \ref{qslemapp},  we have $s_{min}<0$. Using this in \eqref{Aextremesign}, we obtain $A(t_q)>0$. From \eqref{etaatq0}, this implies $\eta(t_q) >0$. An entirely analogous argument shows that when $s$ achieves maximum, $s$ and $A$ switch signs in such a way that  $\eta$ is again positive.

Keeping this observation in mind, our remaining task is to ensure that $\eta$ remains strictly positive at other times apart from those  where $s$ attains its extremal values. To this end, we construct two auxiliary functions that form an envelope (or a cloud) around $\eta$. This will enable us to establish the desired  positivity of $\eta$.

We proceed as follows. Suppose $q_0,s_0,A_0$ are given. These determine the functions $q,s,A$, which are unknowns of a closed ODE system, whose properties have already been analyzed. The choice of $A_0$ also fixes a linear relation between $\eta_0$ and $w_0$. Given these three functions, our aim is to derive a condition involving exactly one of $\eta_0$ or $w_0$. Either of these conditions will ensure that  $\eta(t)$ remains positive. Once such a condition is imposed on $\eta_0$ (or $w_0$), the corresponding constraint on $w_0$ (or $\eta_0$) through $A_0$ follows automatically from  \eqref{Aexp}. For example, if 
$$
\delta_1<\eta_0<\delta_2,
$$
then, using \eqref{Aexp} with the fixed $q_0,s_0,A_0$, we obtain  
$$
\delta_1 < \frac{q_0w_0-A_0}{ks_0}<\delta_2,
$$
which gives the subsequent bounds on $w_0$ as well.
Thus, our task reduces to determining the appropriate conditions.
%Therefore, it remains to find the appropriate conditions.

From Sections \ref{secqs} and \ref{secblowup}, we know that once the initial data $q_0,s_0,A_0$ is given, the  functions $s,q,A$ are completely determined, periodic, and uniformly bounded. Then \eqref{simplestODE} is an inhomogeneous, linear $2\times 2$ system with uniformly bounded coefficients, and therefore solutions $\eta,w$ exist for all $t\in(-\infty,\infty)$. We now state and prove several lemmas.
\begin{lemma}
\label{lemequivalence}
Consider the first order linear ODE,
\begin{align}
\label{firstordereta}
& \tilde\eta' = \frac{A+k\tilde\eta s}{q}.
\end{align}
Let $\mathbb{D}=\{t\in(-\infty,\infty):q(t)=0\}$. For all intervals $I$ with $I\subset\mathbb{D}^c$, suppose $\tilde\eta_1$ and $\tilde\eta_2$ satisfy \eqref{firstordereta} and $\tilde\eta_1(t_\ast)\neq\tilde\eta_2(t_\ast)$ for some $t_\ast\in \mathbb{D}^c$. Then following statements hold,
\begin{enumerate}
    \item $\tilde\eta_1,\tilde\eta_2$ satisfy
    \begin{align}
    \label{etasecondorder}
    & \tilde\eta''+ k\tilde\eta(c+(N-1)s) = k\Gamma^{N-1},
    %& \tilde\eta'' -q\tilde\eta' + k\tilde\eta(c+Ns) = -A + ke^{-(N-1)\int_0^t q}, 
    \end{align}
    for all $t\in(-\infty,\infty)$,
    \item $\tilde\eta_1(t) = \tilde\eta_2(t) = -\frac{A(t)}{ks(t)}>0$ for all $t\in\mathbb{D}$,
    \item $\tilde\eta_1(t)\neq \tilde\eta_2(t)$ for all $t\in\mathbb{D}^c$,
    \item $\tilde\eta'_1(t) \neq \tilde\eta_2'(t)$ for all $t\in\mathbb{D}$. In particular, $\tilde\eta_1 - \tilde\eta_2$ changes sign at only and all $t\in\mathbb{D}$.
\end{enumerate}
\end{lemma}
\begin{proof}
Taking derivative of \eqref{firstordereta}, and using \eqref{qode}, \eqref{sode}, \eqref{Aode} results in \eqref{etasecondorder}, which is linear with bounded coefficients. Note that since $\mathbb{D}$ is discrete, $\tilde\eta_1,\tilde\eta_2$ satisfy \eqref{etasecondorder} for all $t$ by continuity. Hence, the first statement holds. 

Consequently, $\tilde\eta_1,\tilde\eta_2$ are well-defined for all $t$. Therefore, the limit in \eqref{firstordereta},
$$
\lim_{t\to t_q} \tilde\eta_i'(t),\quad t_q\in\mathbb{D},\ i=1,2,
$$
must exist. Hence, 
$$
A(t_q) + k\tilde\eta_i(t_q)s(t_q) = 0, \quad t_q\in\mathbb{D}.
$$
The argument showing that $\tilde\eta_i(t_q)>0$ is identical to the one used in the paragraph following \eqref{etaatq0}. This completes the proof of the second assertion.

To prove the third assertion, we proceed by contradiction. Suppose that 
$$
\tilde\eta_1(\tau_\ast) = \tilde\eta_2(\tau_\ast)
\quad \text{for some}\quad  \tau_\ast\in\mathbb{D}^c.
$$
Note that $\tilde\eta_1,\tilde\eta_2$ both satisfy \eqref{firstordereta} as well as \eqref{etasecondorder}. Consider the initial value problem  (IVP) \eqref{etasecondorder}  with initial data 
$$
\tilde\eta(\tau_\ast) = \tilde\eta_1(\tau_\ast), \quad \tilde\eta'(\tau_\ast) = \frac{A(\tau_\ast) + k\tilde\eta_1(\tau_\ast)s(\tau_\ast)}{q(\tau_\ast)}.
$$
This IVP has a unique solution; hence $\tilde\eta_1\equiv\tilde\eta_2$. However, this is a contradiction since $\tilde\eta_1(t_\ast)\neq \tilde\eta_2(t_\ast)$. Hence, the third assertion follows. 

For $t_q\in\mathbb{D}$, the inequality $\tilde\eta_1'(t_q)\neq\tilde\eta_2'(t_q)$ follows from the second assertion and an ODE uniqueness argument to the one as above. Consequently, if 
$$
\tilde\eta_1(t)-\tilde\eta_2(t)<0 \quad \text{for}\quad  t\in(t_q-\epsilon,t_q)
$$
for  sufficiently small $\epsilon$, then from the second and third assertions, we must have 
$$
\tilde\eta_1'(t_q)-\tilde\eta_2'(t_q)> 0.
$$
Indeed, since 
\begin{align*}
& -\tilde\eta_1(t)>-\tilde\eta_2(t),\quad t\in(t_q-\epsilon,t_q),
\end{align*}
we obtain 
\begin{align*}
& \frac{\tilde\eta_1(t_q)-\tilde\eta_1(t)}{t_q-t}>\frac{\tilde\eta_2(t_q)-\tilde\eta_2(t)}{t_q-t},\quad \text{by Assertion } 2,
\end{align*}
which gives 
\begin{align*}
& \tilde\eta_1'(t_q)\geq \tilde\eta_2'(t_q),
%& \tilde\eta_1'(t_q)> \tilde\eta_2'(t_q),\quad \text{by Assertion } 3.
\end{align*}
and strict inequality follows from Assertion 3.
Hence, $\tilde\eta_1-\tilde\eta_2$ changes sign at $t_q$. This completes the proof.
\end{proof}
Figure \ref{figlemeq} below gives an illustration of Lemma \ref{lemequivalence}.

\begin{figure}[h!] 
\centering
\input{TikzCode/Fig-lemUnique.tex}
\caption{Visualization of Lemma \ref{lemequivalence} with $k=1, c=1, q_0=0.1, s_0=-0.15,A_0=0.2$. The top figure has $N=5$ (generic solutions not periodic) and the bottom with $N=4$ (periodic solutions). }
\label{figlemeq}
\end{figure}

\begin{lemma}
\label{lemeta0}
Suppose that $\tilde\eta$ satisfies \eqref{firstordereta} on the intervals specified in Lemma \ref{lemequivalence} and assume $\tilde\eta(0)>0$. Let $t_c>0$ be the first time (if it exists) such that $\tilde\eta(t_c)=0$. Let 
$$J:=(t_q^1,t_q^2)$$ 
be the smallest interval such that $t_c\in J$ and $q(t_q^i)=0$, $i=1,2$ ($t_1^q$ may be negative). Let $t_A\in J$ be the unique time when $A(t_A)=0$. Then necessarily  
$$
t_c\in (t_q^1,t_A]\quad \text{and}\quad  \tilde\eta(t)<0 \quad \text{for}\quad  t\in(t_c,t_A].
$$
In particular, if $\tilde\eta(t_A)>0$, then such a time $t_c$ does not exist, and $\tilde\eta(t)>0$ for all $t\in [t_q^1,t_q^2]$. Moreover, if $\tilde\eta(t_A)=0$, then 
$$
\tilde\eta(t)>0 \quad \text{for all}\quad  t\in [t_q^1,t_q^2]\backslash\{t_A\}.
$$
\end{lemma}
Note that, by \eqref{tildesode} and \eqref{Aextremesign}, $A$ and $q$ cannot vanish simultaneously.  Moreover, from \eqref{expqex} and \eqref{Aeq}, there exists a unique time 
$$
t_A\in (t_q^1,t_q^2)
$$
such that $A(t_A)=0$. Therefore, $t_A\in J$ exists and is unique. It is possible that $t_A \leq 0$. In that case, $\tilde\eta(t)>0$ for 
$$
t\in[0,t_q^2]\subset [t_A,t_q^2].
$$
\begin{proof}
From Lemma \ref{lemequivalence}, we have $\tilde\eta(t_q^1)>0$. Hence, for some $\epsilon>0$ small enough such that 
$$
t_q^1+\epsilon < t_A,
$$
it follows that $\tilde\eta(t)>0$ for all $t\in[t_q^1,t_q^1+\epsilon]$. The fucntion $\tilde\eta$ satisfies \eqref{firstordereta} in the interval $[t_q^1+\epsilon,t_q^2)$. Multiplying \eqref{firstordereta} by the appropriate integrating factor, we obtain,
$$
\left( \tilde\eta e^{-k\int^t\frac{s}{q}} \right)' = \frac{A}{q} e^{-k\int^t\frac{s}{q}}.
$$
From the direction of flow (clockwise in Figure \ref{figqsphase}), we deduce that $A(t)$ and $q(t)$ have opposite signs in the interval $[t_q^1+\epsilon, t_A)$ and the same sign in $(t_A,t_q^2)$. Indeed, if $t_q^1$ is such that $s(t_q^1) = s_{max}$, then by \eqref{expqex} and \eqref{Aeq}, we have $A(t)<0$ for $t\in[t_q^1+\epsilon, t_A)$ and strictly positive on $(t_A,t_q^2)$. By assertion three of Lemma \ref{qsregularity}, $q(t)>0$ in $J$ and hence, signs of $A(t),q(t)$ are opposite. The same conclusion holds when $s(t_q^1)=s_{min}$.

Therefore, the quantity 
$$
\tilde\eta e^{-k\int^t\frac{s}{q}}
$$
is decreasing in $[t_q^1+\epsilon, t_A)$ and increasing in the interval $(t_A,t_q^2)$. Since 
$$
\left.\tilde\eta e^{-k\int^t\frac{s}{q}}\right|_{t_q^1+\epsilon}>0,
$$
we conclude that 
$$
t_c\in (t_q^1+\epsilon,t_A]\subset (t_q^1,t_A],
$$
and $\tilde\eta(t)<0
$
for $t\in(t_c,t_A]$ follows directly from the ODE above.

If $\tilde\eta (t_A)>0$, then no such $t_c$ exists: otherwise, the argument above would force  $\tilde\eta(t_A)<0$,  which is a contradiction. Hence, $\tilde\eta (t)>0$ for all $t\in(t_q^1,t_q^2)$ and $\tilde\eta$ is positive at $t_q^1,t_q^2$ by Lemma \ref{lemequivalence}. 

Finally, if $t_c=t_A$, then \eqref{firstordereta}  yields $\tilde\eta'(t_A)=0$,  and \eqref{etasecondorder} implies  $$
\tilde\eta''(t_A)>0.
$$
Hence, $t_c=t_A$ is a strict local minimum, and $\tilde\eta$ is positive in a neighborhood of $t_c$. However, for $t>t_A$, we know $A,q$  have same sign,  and the ODE above implies that $\tilde\eta>0$ after $t_A$.
\end{proof}

At this point, a direct  calculation shows that 
\begin{align}
\label{eq:equilsoln}
& \tilde\eta_P := (N\Gamma \tilde s_0)^{-1},
\end{align}
is a strictly positive particular solution of \eqref{etasecondorder}. hus, setting
$$
\hat\eta:= \tilde\eta - \tilde\eta_P, 
$$
we reduce \eqref{etasecondorder} to the homogenous system
\begin{align}
\label{eq:homogen}
\hat\eta '' + k\hat\eta (c+s(N-1)) = 0.
\end{align}
Equation \eqref{eq:homogen} is a linear second order ODE with periodic coefficients. Such equations fall under the scope of  Floquet Theory \cite{YS71}, which addresses general linear ODEs with periodic coefficients. In our situation, the equation enjoys additional structure: \eqref{eq:homogen} is a particular example of the Hill's differential equation \cite{E73}.
%A Galilean transformation in time variable, $t\rightarrow t-t_s^1$ makes $s$ (equivalently, the coefficient in \eqref{eq:homogen}) even. Moreover,
A fundamental set of linearly independent solutions of such an equation must take one of the following forms: 
\begin{enumerate}
\item $\{ e^{mt} p_1(t), e^{-mt} p_2(t) \}$ OR  
\item $\{ e^{imt} p_1(t), e^{-imt} p_2(t) \}$ OR 
\item $\left\{ e^{\left(m+ \frac{i\pi }{T} \right)t} p_1(t), e^{\left(-m + \frac{i\pi }{T} \right)t } p_2(t) \right\}$ OR
\item $\{ p_1(t), t p_1(t) + p_2(t) \}$, 
%\item $\{e^{\frac{i\pi t }{T} } p_1(t), e^{\frac{i\pi t }{T} } (t p_1(t) + p_2(t)) \}$,
\end{enumerate}
for some $m\geq 0$, where $p_1,p_2$ are $T$-periodic functions. 

In the recent work of \cite{CS23}, the authors show that for $N\neq 4$, \eqref{eq:homogen} has a nontrivial unbounded solution, implying that the subcritical region has zero measure in the initial phase space. Consequently, for $N\neq 4$, case (2) above cannot occur.
Nevertheless, one can still determine a set of initial data that guarantees positivity of  $\eta$ on the interval  $[0,T)$, which will be essential for establishing our final result. 
%This will be useful to obtain the final result.

We now define two auxiliary  functions via an IVP using the linear differential equation \eqref{etasecondorder}, using carefully chosen initial conditions. These functions will form a “cloud’’ around the solution $\eta$. 

Let $t_A^i, i=1,2$,  with $0\leq t_A^1<t_A^2$,  be the first two times when 
$$
A(t_A^i)=0.
$$ 
Define the two functions $\eta_i, i=1,2$ as follows,
\begin{align}
\label{eta1}
\begin{aligned}
%& \eta_i'' + k\eta_i(c+(N-1)s) = k \Gamma^{N-1},\\
& \eta_i'' + k\eta_i(c+(N-1)s) = k \Gamma^{N-1},
\end{aligned}
\end{align}
subject to the initial conditions 
$$
\eta_i (t_A^i)=0,\qquad \eta_i'(t_A^i) = 0.
$$

\begin{figure}[ht!] 
\centering
\input{TikzCode/Fig-Lem_distinct.tex}
\caption{Left figure: when $\eta_1\equiv\eta_2$. Right figure: when $\eta_1,\eta_2$ are distinct. }
\label{figpropdistinct}
\end{figure}

\begin{proposition}
\label{propetaipos}
For each $i=1,2$,
$$
\eta_i(t)>0,\quad t\in [0,T]\backslash\{t_A^i\}.
$$
In particular, $\eta_i(0)>0$ and the two functions
$\eta_q(t)$ and $\eta_2(t)$ are distinct.
\end{proposition}
\begin{proof}
We first claim that for any  $A_0$ satisfying condition \eqref{gsseccond}, there exists a solution to \eqref{etasecondorder},  different from  the particular solution \eqref{eq:equilsoln}, that is strictly positive on $[0,T]$. This is straightforward from continuous dependence of solutions in linear ODEs. 

We now establish the proposition. By Lemma  \ref{lemeta0}, 
$$
\eta_1(t)>0\quad \text{for}\quad  t\in [t_q^1,t_q^2]\backslash\{t_A^1\}.
$$
Hence, it can only be zero, if at all,  in the interval $(t_q^2,t_q^3)$. If $\eta_1$ merely touches zero in $(t_q^2,t_q^3)$, then by \eqref{firstordereta} that zero must occur  at $t_A^2$. In that case, \eqref{etasecondorder} together with  uniqueness of solutions to ODEs would force $$
\eta_1\equiv\eta_2,
$$
contradicting the fact that their initial conditions differ.

If, on the other hand,  $\eta_1$ crosses zero in $(t_q^2,t_q^3)$, then Lemma \ref{lemeta0} implies that  $t_A^2$ must lies strictly in between the two roots of $\eta_1$. 
A completely analogous argument holds for  $\eta_2$.
Both possible configurations are illustrated in Figure \ref{figpropdistinct}.

In either case, Lemma \ref{lemequivalence} implies that there can never be a strictly positive solution. However, \eqref{eq:equilsoln} is a strictly positive solution. Hence, none of these two cases are possible and it must be that for each $i=1,2$ $\eta_i$ is positive everywhere on $[0,T]$ except at $t_A^i$, as claimed. 
This finishes the proof.
\end{proof}

The key results are as follows,
\begin{proposition}
\label{etacover1}
Suppose $q_0\neq 0$. If
$$
\min\{\eta_1(0),\eta_2(0)\}< \eta_0 < \max\{\eta_1(0),\eta_2(0)\},
$$ 
then, 
$$
\min\{\eta_1(t),\eta_2(t)\}< \eta(t) < \max\{\eta_1(t),\eta_2(t)\},\quad t\in\mathbb{D}^c.
$$ 
Conversely, if 
$$
\eta_0\notin \left(\min\{\eta_1(0),\eta_2(0)\}, \max\{\eta_1(0),\eta_2(0)\} \right),
$$ 
then
$$
\eta(t)\notin  \left(\min\{\eta_1(t),\eta_2(t)\}, \max\{\eta_1(t),\eta_2(t)\} \right),\quad t\in\mathbb{D}^c.
$$
Here, $\mathbb{D}$ is the same as in the statement of Lemma \ref{lemequivalence}.
\end{proposition}
\begin{proof}
First, observe that $\eta_i's$ satisfy \eqref{firstordereta} together with $\eta$. Indeed, if a function satisfies \eqref{firstordereta} with $\tilde\eta(t_A^1)=0$, then 
$$
\tilde\eta'(t_A^1) = \frac{A(t_A^1) + k\tilde\eta(t_A^1)s(t_A^1)}{q(t_A^1)} = 0,
$$
since $A(t_A^1)=0$. By the first assertion of Lemma \ref{lemequivalence} and uniqueness of solutions to ODEs, we must have $\tilde\eta\equiv\eta_1$. An identical argument applies to $\eta_2$.

Next,  the solution $\eta$ of \eqref{simplestODEeta} also satisfies \eqref{firstordereta}. This follows directly from the definition of $A$ in \eqref{Aexp} and the form of  \eqref{simplestODEeta}. Consequently, the three functions, $\eta,\eta_1,\eta_2$ each satisfy the hypothesis of Lemma \ref{lemequivalence} with $t_\ast = 0$. 

The conclusion now follows immediately from Lemma \ref{lemequivalence}. At every $t\in\mathbb{D}$, the graphs of $\eta_1,\eta_2,\eta$ intersect, which ensures that $\eta$ remains between $\eta_1,\eta_2$ if and only if it it satisfies this property initially
\end{proof}

An illustration of the situation is provided in Figure \ref{figmainprop}.

\begin{figure}[h!] 
\centering
\input{TikzCode/Fig-positivity.tex}
\caption{Example of a function $\widetilde\eta$ trapped between $\eta_i$'s, $k=1,c=1$. Top figure is with $N=5,q_0=0.05,s_0=-0.1,A_0=0.3$ and the bottom is with $N=4,q_0=0.1,s_0=-0.15,A_0=0.5$.}
\label{figmainprop}
\end{figure}

\begin{proposition}
\label{etacover2}
Suppose $q_0= 0$. If
$$
\eta_1'(0)< w_0 < \eta_2'(0),
$$ 
then for any $t>0$, 
$$
\min\{\eta_1(t),\eta_2(t)\}< \eta(t) < \max\{\eta_1(t),\eta_2(t)\},\quad  t\in\mathbb{D}^c.
$$ 
Conversely, if 
$$
w_0\notin \left(\eta_1'(0), \eta_2'(0) \right),
$$ 
then
$$
\eta(t)\notin  \left(\min\{\eta_1(t),\eta_2(t)\}, \max\{\eta_1(t),\eta_2(t)\} \right),\quad  t\in\mathbb{D}^c.
$$
\end{proposition}
\begin{proof}
The proof is very similar to that of Proposition \ref{etacover1} and follows from Lemma \ref{lemequivalence}, only that here the starting time $t=0$ is when the functions $\eta_1,\eta_2,\eta$ cross each other. In other words, $0\in\mathbb{D}$.
\end{proof}

As noted earlier, when $N\neq 4$, equation \eqref{etasecondorder} admits an unbounded solution. In contrast, for $N=4$, numerical evidence strongly indicates that all nontrivial solutions of \eqref{eq:homogen} are $T$-periodic; see Figure \ref{fig:PerAssumpn}. This behavior would, in turn,  yield a subcritical region of positive measure. Motivated by this observation, and for the purposes of the present section, we henceforth restrict attention to the case $N=4$. 

\begin{figure}[h!] 
\centering

\input{TikzCode/Fig-numericAssumpn}
\caption{Examples for $N=4$.}
\label{fig:PerAssumpn}
\end{figure}

At this stage, we do not provide a proof of the $T$-periodicity of the solution of \eqref{etasecondorder}. Instead, we describe the procedure for obtaining the sharp critical threshold under the assumption that all solutions of
\eqref{eq:homogen} are $T$-periodic when $N=4$. A straightforward computation, using the definitions of $\eta$ in \eqref{etawtransf} shows that this assumption is equivalent to Assumption \ref{perassump}. 

\begin{corollary}
\label{maincorol}
If the fundamental matrix of \eqref{etasecondorder} is $T$-periodic, then $\eta(t)>0$ for all $t>0$ if and only if one of the following holds,
\begin{itemize}
    \item If $q_0\neq 0$ then, 
    $$
    \min\{\eta_1(0),\eta_2(0)\}< \eta_0 < \max\{\eta_1(0),\eta_2(0)\}.
    $$
    \item If $q_0 = 0$ then,
    $$
    \eta_1'(0)< w_0 < \eta_2'(0).
    $$
\end{itemize}
\end{corollary}
\begin{proof}
The result for $[0,T)$ follows from Propositions \ref{etacover1} or \ref{etacover2}, followed by an application of Proposition \ref{propetaipos}. The assumption of periodicity extends it for all times.
\end{proof}

Using the results developed above, we move on to proving Theorem \ref{ctcn}.\\

\textit{Proof of Theorem \ref{ctcn}:} Assume that the initial data satisfy the hypothesis of the Theorem. Along each characteristic path \eqref{chpath}, this translates to the condition that for all $\beta>0$,  
$$
(\beta,u_0(\beta),\phi_{0r}(\beta),u_{0r}(\beta),\rho_0(\beta))\in\Theta_4.
$$
We restrict our attention to a single characteristic path and 
suppress explicit dependence on $\beta$ in the initial data. 
%replace the initial data notations with $(\beta,u_0,\phi_{0r},u_{0r},\rho_0)$. 
Under the transformation \eqref{vartransf},  the relevant unknowns become those of the ODE system \eqref{fullodesys}, namely $(q,s,p,\rho)$ with initial data $$(q(0),s(0),p(0),\rho(0)) = \left( \frac{u_0}{\beta},-\frac{\phi_{0r}}{\beta},u_{0r},\rho_0 \right).
$$
Global-in-time existence of these transformed variables is equivalent to the global-in-time existence of the original variables. 

If $\rho(0)= 0$, then as noted via simple calculations at the beginning of Section \ref{secblowup}, then density $\rho$  blows up in finite time. Hence, we can safely assume $\rho(0)>0$.  Applying the transformations  \eqref{etawtransf}, we work instead with the unknowns $(q,s,\eta,w)$ where initial data 
$$
(q(0),s(0),\eta(0),w(0))= (q(0),s(0),p(0)/\rho(0),1/\rho(0)).
$$
Using \eqref{vartransf} and \eqref{etawtransf}, we note that indeed $a=q(0)w(0)-k\eta(0)s(0)$. Set $A(0)=a$ to obtain $A(t)$ as defined in \eqref{Aexp} and satisfying \eqref{Aode}.
Turning to Definition \ref{defthetan}, and using \eqref{tildestrans}, we identify $y_{m}=y_1,y_{M}=y_2$ as in \eqref{tildesboundeq} to get,
$$
A(0) \in \frac{k}{2}\left( \sqrt{ \frac{y_1}{\tilde s_0} }-1, \sqrt{ \frac{y_2}{\tilde s_0}}-1 \right).
$$
Rewriting condition \eqref{gsseccond} in terms of the original variables gives the desired form. Note that $u_0 =  0$ if and only if $q(0)= 0$. In Definition \eqref{aint}, this corresponds precisely to whether $x=0$. 
%this is equivalent to whether $x$ is zero or not. 

If $x\neq 0$ (equivalently $q(0)\neq 0$), then applying the transformation \eqref{etawtransf}, condition \eqref{aint} becomes 
$$
\min\{ \eta_1(0),\eta_2(0) \}  <\eta(0) < \max\{ \eta_1(0),\eta_2(0) \} ,
$$
which is exactly the hypothesis of Proposition \ref{etacover1}. 

If $x= 0$ (equivalently $q(0)= 0$), then using \eqref{vartransf} and \eqref{etawtransf}, \eqref{aint} reduces to
$$
w(0) \in -\frac{ks(0)\eta(0)}{A(0)}\ (\eta_1'(0), \eta_2'(0)).
$$
From $q(0)=0$ and \eqref{Aexp}, the above inclusion simplifies to 
$$
w(0)\in (\eta_1'(0), \eta_2'(0)),
$$
which is exactly the hypothesis of Proposition \ref{etacover2}. 

Along any given characteristic path, the functions $\eta_1$ and $\eta_2$ are completely determined, because $A(0),q(0),s(0)$ are fixed once the initial data are fixed. 
Thus, one can first compute  $\eta_1(0),\eta_2(0)$,  and $\eta_1'(0),\eta_2'(0)$ and then check which hypothesis applies depending on whether  $q(0)=0$ or not.
%In particular, $\eta_1,\eta_2$ can first be evaluated and then it can be checked that one of the above hypothesis is satisfied depending on whether $q(0)=0$ or not.

By Corollary \ref{maincorol}, we conclude that $\eta(t)$ remains positive for all $t>0$. By Proposition \ref{rhoptogether}, all unknowns $(q,s,p,\rho)$ in the system \eqref{fullodesys} then exist for all all time. Since the above argument applies along every characteristic path,
 Lemma \ref{radialfield} combined with  Theorem \ref{local} yields a  global-in-time smooth solution to \eqref{mainsys}.

Conversely, suppose there is a characteristic path corresponding to some parameter $\beta^\ast >0$ such that,
$$
(\beta^\ast, u_0^\ast,\phi_{0r}^\ast,u_{0r}^\ast,\rho_0^\ast):=(\beta^\ast,u_0(\beta^\ast),\phi_{0r}(\beta^\ast),u_{0r}(\beta^\ast),\rho_0(\beta^\ast))\notin\Theta_4.
$$
Without loss of generality, we assume $\rho_0^\ast>0$. Two possibilities arise: \\ 
1. The inclusion in Definition \ref{defthetan} fails, or \\ 
2. The inclusion holds, but the condition  \eqref{aint} fails. \\ 
%Then there could be two situations. Either the inclusion in Definition \ref{defthetan} does not hold or the inclusion holds but $(\beta^\ast, u_0^\ast,\phi_{0r}^\ast,u_{0r}^\ast,\rho_0^\ast)$ violates \eqref{aint}. 
Suppose the first case occurs. As argued earlier, this implies 
$$
A_0^\ast \notin \frac{k}{N-2}\left( \left( \frac{y_1}{\tilde s_0} \right)^{\frac{N-2}{N}}-1, \left( \frac{y_2}{\tilde s_0} \right)^{\frac{N-2}{N}}-1 \right).
$$
Using \eqref{kappaexp} in the above expression, and subsequently with the help of Corollary \ref{ftbresult1}, we conclude that the density must blow up in finite time. 

Now suppose the above inclusion holds and \eqref{aint} fails. 
Arguing in the same manner as above, we find: 
\begin{itemize}
\item  If  $q(0)\neq 0$,
$$
\eta(0)\notin\left(\min\{ \eta_1(0),\eta_2(0) \},   \max\{ \eta_1(0),\eta_2(0) \}\right) ,
$$
\item If $q(0)=0$,
$$
w(0)\notin (\eta_1'(0), \eta_2'(0)).
$$
\end{itemize}
By Corollary \ref{maincorol}, $\eta(t)$ must then reach zero at some finite time $t_c>0$. From Proposition \ref{rhoptogether}, this forces a breakdown of the solution  at $t=t_c$ and smoothness is lost. This completes the proof of the Theorem. \qed

\section{The zero background case}
\label{secc0}
The zero background case has been studied extensively in the literature. Notably, the authors of \cite{WTB12} obtained a sharp critical threshold condition, however, assuming that the flow is expanding ($u_0>0$). A more refined analysis appears in \cite{Tan21}, although the resulting threshold condition was not sharp. 

In this section, we derive a sharp characterization of the subcritical and supercritical regions for general velocity. To this end, we consider system \eqref{fullodesys} with $c=0$,
\begin{subequations}
\label{fullodesys2}
\begin{align}
& \rho' = -(N-1)\rho q - p\rho, \label{rhoode2}\\
& p' = -p^2 - k(N-1)s + k\rho, \label{pode2}\\
& q' = ks-q^2, \label{qode2}\\
& s' = -Nqs, \label{sode2}
\end{align}
\end{subequations}
with the same notation for initial data as before.  Recall also system \eqref{simplestODE}. With zero background, it reduces to,
\begin{subequations}
\label{c0simplestODE}
\begin{align}
& \eta' = w, \label{c0simplestODEeta}\\
& w' = -ks(N-1)\eta + k\Gamma^{N-1}. \label{c0simplestODEw}
\end{align}
\end{subequations}
We will again make use of the quantity $A$ defined in \eqref{Aexp}. A direct computation shows that, for $N >2$,  the expression for $A$ is identical to \eqref{Aeq}. Likewise,  for $N=2$,  the corresponding quantity $B$ coincides with
\eqref{A2eq}. 

A detailed analysis of the $q,s$-subsystem for zero background was carried out in \cite{Tan21}. We state below only the results that will be used directly.

\begin{proposition}[\cite{Tan21}(Theorem 3.5, Lemma 3.4, 3.7)]
\label{tanconv}
Consider the $2\times 2$ ODE system \eqref{qode2}-\eqref{sode2}, and suppose $s_0>0$. Then: 
\begin{itemize}
\item $(q(t),s(t))\to (0,0)$ as $t\to\infty$. Moreover, if $q_0<0$ there exists a unique time $t_q$ such that 
$$
q(t_q)=0,  \quad s(t_q)=\max_{[0,\infty)}s(t)
$$
and $q(t)>0$ for all $t>t_q$.
\item There exists $T>0$ large enough so that, for all $t\geq T$, 
\begin{subequations}
\label{qsconvrate}
\begin{align}
C^q_m(t+1)^{-1} \leq & q(t) \leq C^q_M (t+1)^{-1},\label{qconvrate}\\
C_m^s(t+1)^{-N} \leq & s(t) \leq C_M^s(t+1)^{-N},\qquad  N\geq 3,\label{snconvrate}\\
C_m^s(t+1)^{-2}(1+\ln(t+1))^{-1} \leq & s(t) \leq C_M^s(t+1)^{-2}(1+\ln(t+1))^{-1},  N=2, \label{s2convrate}
\end{align}
\end{subequations}
where $C^q_m,C^q_M,C_M^s,C_m^s$ are positive constants depending on $T,N,q_0,s_0$.
\end{itemize}
\end{proposition} 
As in the nonzero-background case, the Poisson forcing 
prevents concentration at the origin. In particular, Corollary
 \ref{qsalltime} remains valid,  and we must have $s_0>0$. We restate this as the following lemma. 
\begin{lemma}
\label{qsalltimec0}
The functions $q(t)$ and $s(t)$ in system \eqref{qode2}- \eqref{sode2} exist for all $t\geq 0$. 
\end{lemma}
We now turn to several key results. We begin by giving a sufficient condition under which $\eta$ reaches zero in finite time.
\begin{proposition}
\label{c0ftb}
Suppose $N\geq 3$. If 
$$
A_0+\frac{k}{N-2}<0,
$$
then there exists a finite time $t_c>0$ such that $\eta(t_c)=0$.
\end{proposition}
\begin{proof}
Since $\eta$ satisfies \eqref{firstordereta}, we have,
\begin{align}
\label{etaexpode}
\left( \eta e^{-k\int^t\frac{s}{q}} \right)' = \frac{A}{q} e^{-k\int^t\frac{s}{q}} .
\end{align}
Let $t_1>0$ be large enough so that the decay estimates of  Proposition \ref{tanconv} hold for all $t\geq t_1$.  Throughout the proof, constants  
$$
0<C_m<C_M
$$
 may change from line to line but only depend on $k,N,q_0,s_0,\eta_0,w_0$ and $t_1$. 
 
 From \eqref{qsconvrate},  for $t_1\leq \tau<t$,
\begin{align*}
%&  k\int_{t_1}^t \frac{s}{q},\\
%C_m \int_{\tau}^t (\xi+1)^{-(N-1)}d\xi \leq & k\int_{\tau}^t\frac{s}{q}  \leq C_M \int_{\tau}^t (\xi+1)^{-(N-1)}d\xi,\\
& k\int_{\tau}^t\frac{s}{q}  \leq C_M \int_{\tau}^t (\xi+1)^{-(N-1)}d\xi \leq C_M
%C_m\left( \frac{1}{(\tau+1)^{N-2}} -  \frac{1}{(t+1)^{N-2}}\right) \leq & k\int_{\tau}^t\frac{s}{q} \leq C_M\left( \frac{1}{(\tau+1)^{N-2}} -  \frac{1}{(t+1)^{N-2}}\right), \numberthis\label{qsconvbounds}\\
%%& k\int_{\tau}^t\frac{s}{q} \leq C_M. \numberthis \label{qsunibounds}
\end{align*}
Integrating \eqref{etaexpode} from $t_1$ to $t$, we have,
\begin{align*}
\eta(t) & = \eta(t_1) e^{k\int_{t_1}^t\frac{s}{q}} + \int_{t_1}^t \frac{A(\tau)}{q(\tau)}e^{k\int_\tau^t\frac{s}{q}}d\tau =: \RN{1} + \RN{2}.
\end{align*}
By \eqref{qsunibounds}, the exponential factor is uniformly bounded for $t\geq t_1$, so $\RN{1}$ is uniformly bounded as well.  

Next, we estimate $A$. Using \eqref{Aeq} and \eqref{expqex}, 
and the decay estimate \eqref{snconvrate}  for $s$, 
%we can find convergence rates for $A$ as well. Using \eqref{snconvrate} in \eqref{Aeq} along with \eqref{expqex}, 
we find that for all $t\geq t_1$,
\begin{align*}
A(t) & = \left( A_0+\frac{k}{N-2}\right) \Gamma -\frac{k}{N-2}\Gamma^{N-1} \leq -C_m (t+1)^{-1} .
\end{align*}
In particular, from \eqref{qconvrate},
$$
\frac{A(t)}{q(t)} \leq -C_m, \quad t\geq t_1.
$$
Since $q,s>0$  and $A<0$ for $t>t_1$, we obtain 
\begin{align*}
\RN{2} & = \int_{t_1}^t \frac{A(\tau)}{q(\tau)}e^{k\int_\tau^t\frac{s}{q}}d\tau 
 \leq \int_{t_1}^t \frac{A(\tau)}{q(\tau)}d\tau 
%& \leq -C_m\int_{t_1}^t  e^{k\int_\tau^t\frac{s}{q}}d\tau\\
%& \leq -C_m e^{-\frac{C_m}{(t + 1)^{N-2}}} \int_{t_1}^t e^{\frac{C_m}{(\tau + 1)^{N-2}}}d\tau\\
 \leq -C_m\int_{t_1}^t d\tau.
\end{align*}
Combining the bounds on $\RN{1}$ and $\RN{2}$, we conclude that for all  $t> t_1$, 
\begin{align*}
\eta(t)& = \RN{1}+\RN{2}  \leq C_M - C_m \int_{t_1}^t d\tau, 
\end{align*}
which becomes negative at some finite time. Thus,  there exists  $t_c>0$ such that $\eta(t_c)=0$. This completes the proof.
\end{proof}
We now move on to characterizing the subcritical region. 
As a first step, we define $\eta_1$ as the solution of the initial value problem obtained from \eqref{firstordereta}, 
\begin{align}
\label{c0eta1}
\eta_{1}' = \frac{A}{q} + \frac{ks}{q}\eta_{1},\quad \eta_{1} (t_A)=0,
\end{align}
where $t_A>0$ is the time (possibly unique) such that $A(t_A)=0$, provided such a time exists. 
\begin{proposition}
\label{c0prop1}
Consider the case $N\geq 3$. Suppose $s_0>0$ and $q_0>0$. Given $\eta_0>0$, we have $\eta(t)>0$ for all $t>0$ if one of the following conditions holds: 
\begin{itemize}
    \item $A_0\geq 0$,
    \item $-\frac{k}{N-2} < A_0 < 0$ and $\eta_0>\eta_{1}(0)$.
\end{itemize}
Conversely, if $-\frac{k}{N-2} < A_0 < 0$ and $\eta_0\leq \eta_{1}(0)$, then there exists a time $t_c>0$ such that $\eta(t_c)=0$.
\end{proposition}
\begin{proof}
Under the hypotheses and Proposition \ref{tanconv},  both $s(t)$ and $q(t)$ remain positive for all $t>0$. From \eqref{sode2}, $s(t)$ is monotonically decreasing to zero. 
By \eqref{expqex}, 
$$
\Gamma(t) = e^{-\int_0^t q}
$$
is strictly decreasing and tends to zero as $t\to\infty$.

If $A_0\geq 0$, then
$$
\kappa \geq 1,
$$
where $\kappa$ is the root of $A(\Gamma)$ as defined in \eqref{kappaexp}. By \eqref{Aeq}, this implies $A(t)>0$ for all $t>0$. Consequently, by
\eqref{etaexpode}, $\eta$ cannot reach zero in finite time if $\eta_0>0$. 

Now suppose the second hypothesis holds. Here, $\kappa < 1$. From the behavior of $s,q,\Gamma,A$ as listed above, it follows that there is a unique time $t_A>0$ such that $A(t_A)=0$ with 
$$
A(t)<0 \quad  \text{for}  \quad t<t_A \quad  \text{and} \quad A(t)>0\quad  \text{for} \quad t>t_A.
$$
By \eqref{etaexpode}, $\eta$ remains positive  for all $t>0$  if $\eta(t_A)>0$. 

In fact, $\eta_{1}(t)$ serves as a lower barrier for $\eta(t)$; see Figure \ref{c0fig1}. Since $\eta,\eta_{1}$ satisfy the same first order linear ODE,  we have the comparison relation 
$$
(\eta-\eta_{1})' = \frac{ks}{q}(\eta - \eta_{1}).
$$
By \eqref{c0eta1},  $\eta_{1}'(t_A)=0$. Differentiating once more, one checks that $\eta_{1}''(t_A)>0$, so $t_A$ is the unique pint of minimum of $\eta_{1}$ is attained. Hence, $\eta(t)>0$ for all $t>0$. 

Conversely, if $\eta_0\leq \eta_{1}(0)$, then  $\eta(t) \leq \eta_1(t)$  for all $t$, and since $\eta_1(t_A)=0$, there must exist a time $t_c \leq t_A$ such that $\eta(t_c=0)$.
%Lastly, we will show that if $A_0 < -\frac{k}{N-2}$, then $\eta$ becomes zero in finite time. To show this,
\end{proof}

\begin{figure}[h!] 
\centering
\input{TikzCode/Fig-ZeroBack1.tex}
\caption{$\eta_1$ and a plausible solution $\eta$ with $k=1,N=5,q_0=2,s_0=1,A_0=-0.3$.}
\label{c0fig1}
\end{figure}

\begin{proposition}
\label{c0prop2}
Consider the case $N\geq 3$. Suppose $s_0>0,q_0=0$,  and  $A_0 >-\frac{k}{N-2} $. Given $\eta_0>0$, we have $\eta(t)>0$ for all $t>0$ if and only if 
$$
w_0>\eta_{1}'(0).
$$
\end{proposition}
Note that $\eta_{1}'(0)$ is in the limit sense in \eqref{c0eta1} since $q(0)=0$. We know that this limit exists because $\eta_{1}$ satisfies \eqref{etasecondorder} (with $c=0$), which is an inhomogeneous second order linear ODE with bounded coefficients. Also, $q_0=0$ implies $A_0<0$.
\begin{proof}
The proof of Proposition \ref{c0prop2} closely parallels the argument used in the second part of Proposition \ref{c0prop1}. A comparison principle applies because solutions to the ODE are unique.  Since both $\eta$ and $\eta_{1}$ satisfy \eqref{etasecondorder} and 
$$
\eta_0 = \eta_{1}(0) = -\frac{A_0}{ks_0},
$$
the condition $\eta'(0)=w_0>\eta_{1}'(0)$ ensures that  $\eta(t)>\eta_{1}(t)$ for all $t>0$. In particular $\eta_1(t) \geq 0$, so $\eta(t)$ stays positive for all time.

The converse also follows immediately: if $w_0 \leq \eta_{1}'(0)$, then $\eta(t) \leq \eta_{1}(t)$ for all $t$, and  since $\eta_{1}(t_A)= 0$, $\eta$ must vanish at some $t_c \leq t_A$.
\end{proof}

We now address the case  $q_0<0$. We take derivative of \eqref{c0eta1} and obtain two second order IVPs for $i=1,2$,
\begin{align}
\label{c0secondordereta1}
\begin{aligned}
& \eta_i'' + k(N-1)s\eta_i = k \Gamma^{N-1},\\
& \eta_i (t_A^i)=0,\qquad \eta_i'(t_A^i) = 0,
\end{aligned}
\end{align}
where each $t_A^i\geq 0$ satisfies $A(t_A^i)=0$. There may be only one such time; in that case, we only consider $\eta_1$. Note that by uniqueness of solutions, $\eta_1$ defined in \eqref{c0eta1} is the same function as $\eta_1$ introduced above. As in \eqref{trajectoryN}, we obtain the trajectory for this case as,
\begin{align}
\label{c0trajectoryN}
& q^2 s^{-\frac{2}{N}} + \frac{2k}{N-2}s^{1-\frac{2}{N}} = R_N,
\end{align}
From this relation, one sees directly that the maximum of $s(t)$ occurs when $q=0$. Consequently, 
\begin{align}
\label{c0smaxexp}
& s_{max} := \left( \frac{R_N(N-2)}{2k} \right)^{\frac{N}{N-2}}.
\end{align}

\begin{proposition}
\label{c0prop3}
Consider the case $N\geq 3$. Suppose $s_0>0$ and $q_0<0$, and let $\eta_0>0$. Then: 
\begin{itemize}
    \item If $A_0\geq 0$, then $\eta(t)>0$ for all $t>0$ if and only if 
    $$
    \kappa < (s_{max}/s_0)^{1/N} \quad \text{and}\quad  \eta_1(0)<\eta_0 < \eta_2(0).
    $$
    \item If $-\frac{k}{N-2} < A_0 <0$,  then $\eta(t)>0$ for all $t>0$ if and only if 
    $$
    \eta_0<\eta_1(0).
    $$
\end{itemize}
\end{proposition}

\begin{proof}
Suppose $A_0\geq 0$. From the first assertion of Proposition \ref{tanconv} and \eqref{Aexp}, we conclude that at $t=t_q$,
$$
\eta(t_q) = -\frac{A(t_q)}{ks(t_q)} = -\frac{A(t_q)}{ks_{max}}.
$$
Therefore, a necessary condition for $\eta$ to be positive is that $A(t_q)<0$. From \eqref{expqex} and \eqref{Aeq}, the condition is equivalent to,
$$
\kappa < \left( \frac{s_{max}}{s_0}\right)^{\frac{1}{N}}.
$$ 
Here, we must keep in mind that the dynamics of $s$ of  is such that increases until $t=t_q$ then decreases monotonically and approaches zero as $t\to\infty$. If the above inequality holds, then there are two positive times, $t_A^i,i=1,2$, when $A(t_A^i)=0$. Also, $t_q\in (t_A^1,t_A^2)$. The very same arguments as in Proposition \ref{propetaipos} allow us to conclude that $\eta_i(t)>0$ for $t\in[0,\infty)\backslash\{t_A^i\}$. 
%in Lemma \ref{lemeta0}, we have that $\eta$ is positive if it is so in the intervals $(0,t_A^1)$ and $(t_q,t_A^2)$. 
Since $\eta_i$'s also satisfy \eqref{c0eta1}, the all-time-positivity of $\eta$ is guaranteed once again by arguments as in Lemma \ref{lemequivalence} if it was in between $\eta_1$ and $\eta_2$ at $t=0$. In particular, assertions 1-4 of Lemma \ref{lemequivalence} are valid with a slight modification that the set $\mathbb{D} = \{t_q\}$ has only one element, see Figure \ref{c0fig2} for a visualization.
Conversely, if $\eta_0\notin (\eta_1(0),\eta_2(0))$, then from Lemma \ref{lemequivalence} $\eta$ becomes zero in finite time.  

If $-\frac{k}{N-2} < A_0 <0$, then $\kappa <1$. From Proposition \ref{tanconv}, $s$ increases to a maximum and then decreases to zero. Once again, making use of the relation \eqref{expqex} and \eqref{Aeq}, we have that $A$ is zero only once, at $t=t_A$. From \eqref{c0eta1}, we have,
$$
\left( \eta_1 e^{-k\int^t\frac{s}{q}} \right)' = \frac{A}{q} e^{-k\int^t\frac{s}{q}}.
$$
We conclude that for $t<t_q$, $\eta(t)>0$ since $A,q$ are both negative in this interval. Since from Lemma \ref{lemequivalence}, $\eta(t)>\eta_1(t)$ for $t>t_q$, we conclude that $\eta(t)>0$ for all $t>t_q$ since $\eta_1$ is nonnegative and serves as a lower bound for $\eta$ in this domain. Conversely, if $\eta_0 >\eta_1(0)$, then $\eta(t)<\eta_1(t)$ for all $t>t_q$ and hence, it must be zero in a finite time, $t<t_A$, since $\eta_1(t_A)=0$. 
\end{proof}

\begin{figure}[h!] 
\centering
\input{TikzCode/Fig-ZeroBack2.tex}
\caption{$\eta_1,\eta_2$ and a plausible solution, $\eta$, with $k=1,N=5,q_0=-2,s_0=1,A_0=1$}
\label{c0fig2}
\end{figure}

\begin{proposition}
\label{c0prop4}
Let $N\geq 3$ and $s_0>0$ with $q_0$ arbitrary. If 
$$
A_0 = -\frac{k}{N-2},
$$
then there exists a time $t_c>0$ such that $\eta(t_c) = 0$.
\end{proposition}
We state this proposition and its proof separately because the technique used in the proof of Proposition \ref{c0ftb} does not apply. In fact, it turns out that it is much easier to work in the $p,\rho$ variables instead of the $\eta,w$ variables.

\textit{Proof of Proposition \ref{c0prop4}:} Using \eqref{Aexp} and $A_0=-k/(N-2)$ in the expression of $A$ as in \eqref{Aeq}, we obtain,
$$
q(t)w(t) - k\eta(t)s(t) = A(t) = -\frac{k}{N-2} \Gamma^{N-1} .
$$
Substituting for $\eta,w$ using \eqref{etawtransf}, we obtain,
$$
\frac{qp-ks}{\rho} = -\frac{k}{N-2}.
$$
As a result, we can find $p$ in terms of the other variables as,
$$
p = \frac{ks}{q} - \frac{k\rho}{(N-2)q}.
$$
We can divide by $q$ because from \eqref{qconvrate}, $q$ is eventually positive and decays accordingly. We will consider sufficiently large times so that this holds. Plugging this in the ODE of $\rho$, \eqref{rhoode2}, we get,
$$
\rho' = \frac{k}{q(N-2)}\rho^2 - \rho\left( q(N-1) + \frac{ks}{q}\right).
$$
For all sufficiently large times, the rates in Proposition \ref{tanconv} hold. We can assume $\rho$ has not already blown up, because if it has then the proof is done. Using these rates, we conclude that $ks/q$, $q$ are bounded. Therefore, there occurs a Riccati-type blow up of density, and the blow up is aggravated by the $q$ in the denominator.    
\qed\\

Propositions \ref{c0ftb}-\ref{c0prop4} 
%\ref{c0prop1}, \ref{c0prop2}, \ref{c0prop3} and \ref{c0prop4} 
enable us to assemble the complete picture for the case $N\geq 3$. 
We now turn to the critical case $N=2$.In what follows, we omit arguments that repeat those already presented.

%We will omit the repetitive parts in the proofs of these results. 
Recall the quantity $B$ defined in \eqref{A2eq}, which plays a role  analogous to $A$. Note that there always exists a positive root of $B(\Gamma)$, 
$$
\kappa = e^{\frac{A_0}{k}}, 
$$
no matter the sign of $A_0$. This differs from the situation for $A$. 
The value of $\kappa$ may lie on either side of $1$, or equal to $1$,  depending on the sign of $A_0$. For the case $N=2$, we consider the function $\eta_{1}$ defined in \eqref{c0eta1} and the functions $\eta_i$'s defined in \eqref{c0secondordereta1},  with $A$ replaced by $B$ in their definitions.  

\begin{proposition}
\label{c0n2prop1}
Suppose $N=2$ and consider $\eta_{1}$ as in \eqref{c0eta1}. Assume  $s_0>0$ and $q_0>0$. Given $\eta_0>0$, we have $\eta(t)>0$ for all $t>0$ provided one of the following conditions is satisfied,
\begin{itemize}
    \item $A_0\geq 0$;
    \item $A_0 < 0$ and $\eta_0>\eta_{1}(0)$.
\end{itemize}
Moreover, if $A_0 < 0$ and $\eta_0\leq \eta_{1}(0)$, then there exists a time $t_c>0$ such that $\eta(t_c)=0$.
\end{proposition}
The proof is almost identical to that of Proposition \ref{c0prop1}.
\begin{proposition}
\label{c0n2prop2}
Consider the case $N=2$. Suppose $s_0>0$ and $q_0=0$. Given $\eta_0>0$, we have $\eta(t)>0$ for all $t>0$ if and only if 
$$
w_0>\eta_{1}'(0).
$$
\end{proposition}
The proof follows the same argument as Proposition \ref{c0prop2}.

\begin{proposition}
\label{c0n2prop3}
Consider the case $N=2$. Suppose $s_0>0$ and $q_0<0$. Given $\eta_0>0$, we have the following:
\begin{itemize}
    \item If $A_0\geq 0$:  $\eta(t)>0$ for all $t>0$ if and only if 
    $$
    e^{\frac{A_0}{k}} < \sqrt{(s_{2,max}/s_0)} \quad \text{and}\quad  \eta_1(0)<\eta_0 < \eta_2(0).
    $$
    \item If $A_0 <0$:  $\eta(t)>0$ for all $t>0$ if and only if $$
    \eta_0<\eta_1(0).
    $$
\end{itemize}
\end{proposition}
Again, the proof parallels that of Proposition \ref{c0prop3}.

The quantity $s_{2,max}$ differs from $s_{max}$ because the trajectory equation takes a different form in the critical case. Similar to the way we obtained \eqref{c0trajectoryN}, we obtain the trajectory for  $N=2$,
\begin{align}
\label{c0trajectory2}
& \frac{q^2}{s}+ k\ln(s) = R_2, 
\end{align}
where 
$$
R_2 = \frac{q_0^2}{s_0} + k\ln(s_0).
$$
Consequently,
$$
s_{2,max} = e^{\frac{R_2}{k}}.
$$
Finally, we examine the case of vanishing initial density, $\rho_0 = 0$. From \eqref{rhoode2}, the condition $\rho_0=0$ implies $\rho\equiv 0$, as long as the solution $p$ of \eqref{pode2} exists. In this situation, \eqref{pode2} reduces to,
\begin{align}
\label{rho0pode2}
p' = -p^2 - k(N-1)s.
\end{align}

\begin{proposition}
\label{proprho0}
Consider system \eqref{fullodesys2} and suppose $\rho_0 = 0$ in \eqref{rhoode2}. Then $p$ remains bounded for all time if and only if 
$$
q_0>0 \quad \text{and} \quad p_0\geq \frac{ks_0}{q_0}.
$$
Furthermore, if these conditions hold, then 
$$
p(t)\geq \frac{ks(t)}{q(t)},\quad t>0.
$$
On the other hand, if $q_0\leq 0$ or $p_0< \frac{ks_0}{q_0}$, then there exists a time $t_c>0$ such that,
$$
\lim_{t\to t_c^-}p(t) = -\infty.
$$
\end{proposition}
\begin{proof} 
Consider the quantity $qp-ks$. From \eqref{rho0pode2}, \eqref{qode2} and \eqref{sode2}, we have,
\begin{align*}
(qp-ks)' & = p(ks-q^2) - q(p^2+k(N-1)s) + kNsq\\
& = kps - pq^2 -qp^2 + kqs\\
& = -(p+q)(qp- ks).
\end{align*}
Consequently, as long as $p$ exists, $qp-ks$ maintains sign and we have,
\begin{align}
\label{qpks}
& qp-ks = (q_0p_0-ks_0)e^{-\int_0^t q}e^{-\int_0^t p}.
\end{align}
{\bf Case 1:}  $q_0>0$ and $p_0\geq \frac{ks_0}{q_0}$. \\
By Proposition \ref{tanconv},  $q(t)>0$ for all $t$. Suppose,  for contradiction, that $p$ breaks down in finite time. Since $s(t)$ is uniformly bounded,  \eqref{rho0pode2} that at the time of breakdown, $p\to -\infty$. However, $q(t),s(t)$ and $e^{-\int_0^t}$ are uniformly bounded. Hence, before the breakdown time, the left-hand-side of \eqref{qpks} would have to become negative but the right-hand-side is non-negative. This is a contradiction and hence, $p(t)$ stays finite for all $t$. In particular,  
$$
p(t)\geq ks(t)/q(t),
$$
with the decay rate following from Proposition \ref{tanconv}.\\
{\bf Case 2:}  $q_0>0$ but $p_0 < \frac{ks_0}{q_0}$. \\
We first analyze the case $N\geq 3$; the critical case  $N=2$ needs separate treatment and will be addressed lateer. 

As long as $p$ remains finite, we have 
$$
k(N-1)\int_t^\infty s(\tau) d\tau \leq p(t)\leq \frac{ks(t)}{q(t)}.
$$
The right-hand inequality follows directly from  \eqref{qpks}.  By Proposition \ref{tanconv}, the integral on the left-hand-side is finite. 

To justify the left inequality, assume for contradiction that there exists $t_1$ such that 
$$
k(N-1)\int_{t_1}^\infty s(\tau)d\tau > p(t_1).
$$
Integrating \eqref{rho0pode2} from $t_1$ to some $t_2> t_1$ then gives 
$$
p(t_2) < p(t_1) - k(N-1)\int_{t_1}^{t_2} s(\tau)d\tau  < 0 .
$$
But once $p(t_2)<0$, equation \eqref{rho0pode2} implies a Riccati-type finite-time blowup, so that $p(t)\to -\infty$ for some  $t>t_2$.This completes the contradiction.

Therefore, we may assume that 
$$
k(N-1)\int_t^\infty s(\tau) d\tau \leq p(t) \quad \text{for}\quad  t>0.
$$
Choose $t_\ast$ large enough so that the convergence estimates of Proposition \ref{tanconv} apply. Using those estimates, we can rewrite the bounds on $p$ as 
\begin{align}
\label{pubbelow}
& 0 < C_1 (1+t)^{-(N-1)}\leq p(t) \leq C_2 (1+t)^{-(N-1)}, \quad t\geq t_\ast
\end{align}
where the positive constants $C_i's$  depend only on $s_0,q_0,t_\ast,k$, and their values may change from line to line.

From Proposition \ref{tanconv}, and \eqref{pubbelow}, we have
$$
0> qp - ks \geq C_1 (1+t)^{-N} - C_2 (1+t)^{-N} \geq -C_1 (1+t)^{-N}.
$$
We now examine the right-hand side of \eqref{qpks}: 
\begin{align*}
(q_0p_0-ks_0)e^{-\int_0^t q}e^{-\int_0^t p} & = -C_2 s^{\frac{1}{N}} e^{-\int_0^t p}\\
& \leq -C_2 (1+t)^{-1}.
\end{align*}
Here we used Lemma \ref{tildesexpq} for the equality. We used \eqref{pubbelow} to conclude that $p$ is integrable and hence, the inequality holds. 

Combining the above bounds for the two sides of \eqref{qpks}, we find that for sufficiently large $t$,
$$
qp-ks \geq -C_1(1+t)^{-N} > -C_2 (1+t)^{-1} \geq (q_0p_0-ks_0)e^{-\int_0^t q}e^{-\int_0^t p},
$$
which is impossible. a contradiction. Therefore, $p$ must blow up in finite time; that is, 
$$
\lim_{t\to t_c^-}p(t) = -\infty \quad \text{for some}\quad  t_c>0.
$$
{\bf The case $q_0\leq 0$.}\\
First observe that $p_0> 0$,  for is $p_0 \leq 0$ then a Riccati-type blowup occurs immediately.  Thus
$$
q_0 p_0 - ks_0 <0,
$$
and hence, $qp-ks<0$ for all $t$. From Proposition \ref{tanconv}, we know that $q(t)$ eventually becomes positive and follow the same  decay rates. Hence, the same contradiction argument applies,  and again $p$ blows up in finite time.\\
{\bf The case $N=2$}. \\
 As  in the case $N\geq 3$, it suffices to derive a contradiction under the assumption
  $$
  q_0>0 \quad \text{and}  \quad p_0<\frac{ks_0}{q_0}.
  $$
  %All the other arguments are the same. To this end, we assume $q_0>0$ and $p_0<\frac{ks_0}{q_0}$. 
Assuming $p(t)$ exists for all time, we have
$$
k\int_t^\infty s(\tau)d\tau \leq p(t)\leq \frac{ks(t)}{q(t)}.
$$
For all $t\geq t_\ast$ ( with $t_\ast$ again chosen so that the decay rates of Proposition \ref{tanconv} hold), this becomes 
\begin{align}
\label{pbounds2woE}
C_1 \int_t^\infty (\tau+1)^{-2}(1+\ln(1+\tau))^{-1}d\tau \leq p(t) \leq C_2 (t+1)^{-1}(1+\ln(1+t))^{-1}.
\end{align} 
We focus on the integral on the left-hand side. With the substitution
$$
u=a+\ln(1+\tau), 
$$
the integral becomes
$$
C_1\int_{1+\ln(1+t))}^\infty \frac{e^{-\tau}}{\tau}d\tau.
$$
This expression can be represented using a well-known exponential integral function given by 
$$
E_1(t) = \int_t^\infty \frac{e^{-\tau}}{\tau}d\tau.
$$
Moreover, it has the following bounds (see \cite[Page 229]{AbSt}): 
$$
\frac{e^{-t}}{2}\ln\left( 1+\frac{2}{t} \right)< E_1(t) < e^{-t}\ln\left( 1+\frac{1}{t} \right).
$$
Using the lower bound above in \eqref{pbounds2woE}, we obtain 
\begin{align}
\label{pbounds2}
& \frac{C_1}{1+t}\ln\left( 1+\frac{2}{1+\ln(1+t)} \right)\leq p(t) \leq \frac{C_2}{1+t} (1+\ln(1+t))^{-1}, \quad t\geq t_\ast.
\end{align}
Therefore, 
\begin{align*}
0>qp - ks & \geq \frac{C_1}{(1+t)^2}\ln\left( 1+\frac{2}{1+\ln(1+t)} \right)- \frac{C_2}{(1+t)^2}\frac{1}{(1+\ln(1+t))}\\
& \geq -\frac{C_2}{(1+t)^2}\frac{1}{(1+\ln(1+t))},
\end{align*}
for all sufficiently large $t$. Next, consider the right-hand-side of \eqref{qpks} for $N=2$. 
By (\ref{Gammaexp}), 
$$
e^{-\int_0^t q}=\frac{\sqrt{s}}{\sqrt{s_0}}. 
$$
Thus, for $t\geq t_*$,  using  
\eqref{s2convrate} and the assumption $q_0p_0-ks_0<0$, we obtain 
$$
(q_0p_0-ks_0)e^{-\int_0^t q}= \frac{(q_0p_0-ks_0)}{\sqrt{s_0}}\sqrt{s} 
\leq  -\frac{C_1}{(1+t)\sqrt{1+\ln(1+t)}}.
$$
%For $t\geq t_\ast$, from \eqref{s2convrate}, we have,
%\begin{align*}
%\frac{(q_0p_0-ks_0)}{\sqrt{s_0}}\sqrt{s} e^{-\int_0^t p} & %\leq -\frac{C_1}{(1+t)\sqrt{1+\ln(1+t)}}e^{-\int_0^t p}.
%\end{align*}
On the other hand, 
\begin{align*}
e^{-\int_0^t p} & = e^{-\int^{t_\ast}_0 p} e^{-\int_{t_\ast}^t p}\\
& \geq C_1  e^{-\int_{t_\ast}^t \frac{C_2}{1+\tau} (1+\ln(1+\tau))^{-1}d\tau}.
\end{align*} 
Since 
$$
\int_{t_*}^t \frac{1}{(1+\tau)(1+ \ln (1+\tau))}d\tau =
\ln(1+\ln(1+t)) -\ln(1+\ln(1+t_*)),
$$
we obtain (absorbing constants),
\begin{align*} 
& e^{-\int_0^t p}  \geq  C_1 e^{-C_2\ln(1+\ln(1+t))}\\
& = \frac{C_1}{(1+t)[1+\ln(1+t)]^{C_2+1/2}}.
\end{align*}
\iffalse 
Modifying $C_1$ to include the contribution of $e^{-\int_0^{t_\ast} p }$, and using $p\leq ks/q$ along with \eqref{qconvrate}, \eqref{s2convrate}, we complete the further calculations,
\begin{align*}
-\frac{C_1}{(1+t)\sqrt{1+\ln(1+t)}}e^{-\int_0^t p} & = -\frac{C_1}{(1+t)\sqrt{1+\ln(1+t)}}e^{-\int_{t_\ast}^t p}\\
& \leq -\frac{C_1}{(1+t)\sqrt{1+\ln(1+t)}}e^{-\int_{t_\ast}^t \frac{C_2}{1+\tau} (1+\ln(1+\tau))^{-1}d\tau}\\
& =  -\frac{C_1}{(1+t)\sqrt{1+\ln(1+t)}}e^{-C_2\ln(1+\ln(1+t))}\\
& = -\frac{C_1}{(1+t)[1+\ln(1+t)]^{C_2}}.
\end{align*}
Note that we modified $C_1$ in each line and to obtain the last equality, we also transformed $C_2\to C_2 + \frac{1}{2}$.
\fi 
Hence,  substituting into \eqref{qpks} we find   
$$
qp - ks =(q_0p_0-ks_0)e^{-\int_0^t q}e^{-\int_0^t p}
\leq  - \frac{C_1}{(1+t)[1+\ln(1+t)]^{C_2+1/2}}.
$$
However, this contradicts the fact that for sufficiently large 
$t$,
$$
qp - ks \geq -\frac{C_2}{(1+t)^2(1+\ln(1+t))} > -\frac{C_1}{(1+t)[1+\ln(1+t)]^{C_2+1/2}}. 
$$
This contradiction completes the argument.
%This completes the proof.
\end{proof}

\textit{Proof of Theorem \ref{c0ftbn}:} Suppose the hypothesis holds. Then by applying \eqref{etawtransf} to Proposition \ref{c0ftb}, we immediately obtain the finite-time blowup of the density whenever 
$$
A_0<-k/(N-2).
$$
If $A_0=-k/(N-2)$, the desired conclusion follows from Proposition \ref{c0prop4}. \qed \\

\textit{Proof of Theorem \ref{c0ctcn}:} Assume the initial data satisfy the hypotheses of the theorem. Along each  characteristic path \eqref{chpath}, this means that for all $\beta>0$,  
$$
(\beta,u_0(\beta),\phi_{0r}(\beta),u_{0r}(\beta),\rho_0(\beta))\in\Sigma_N\cup \{ (\beta,x,y,z,0): x>0, z\geq -ky/x \}.
$$
We analyze a single characteristic path and  simplify the notation by writing the initial data as $(\beta,u_0,\phi_{0r},u_{0r},\rho_0)$. Under the transformation \eqref{vartransf}, the unknowns become the variables of system \eqref{fullodesys2}, namely $(q,s,p,\rho)$. Global-in-time existence of these variables is equivalent to global-in-time existence of the original variables.

If $\rho(0) = 0$, Proposition \ref{proprho0} guarantees global-in-time existence of $p(t)$ and hence, the full solution exists for all time. 

{\bf Case 1:} $\rho(0) > 0$ \\ 
%Next, we prove global existence for the case when $\rho(0) > 0$. 
By Definition \ref{defsigman}, using the equivalence of $a$ with \eqref{Aexp}, we analyze conditions \eqref{aneg}- \eqref{apos2} individually. 
%\eqref{aneg}, \eqref{azero}, \eqref{apos1}, \eqref{apos2} 
Suppose first that 
%To this end, let it be such that,
$$
A_0 \in \left( -\frac{k}{N-2}, 0\right).
$$
This ensures condition \eqref{aneg}. Rearranging \eqref{aneg} using the transformations \eqref{vartransf} and  \eqref{etawtransf}, we obtain for $x\neq 0$,
$$
\frac{\eta_0}{\beta q_0} < \frac{\eta_1(0)}{\beta q_0},
$$
while for $x=0$,
$$
w_0>\frac{d\eta_1 (0)}{dt}.
$$
Since $x=0$ is equivalent to $q_0=0$, the first inequality is precisely the hypothesis of the second assertion of Proposition \ref{c0prop1} or the second assertion of Proposition \ref{c0prop3} (depending on the sign of  $q_0$). The second inequality is the hypothesis to Proposition \ref{c0prop2}. As a result, by \eqref{etawtransf}, the variables $\rho,p$ in \eqref{rhoode2}-\eqref{pode2} exists for all time. 

{\bf Case 2:  $A_0=0$}\\ 
In this case condition \eqref{azero} reduces to 
$$
\eta_0<\eta_2(0) \quad \text{for}\quad  q_0<0,
$$
and imposes no additional constraints when $q_0>0$. Note that $q_0$ cannot be equal to zero because that would be a violation to $A_0=0$. These two scenarios correspond exactly to the hypothesis of the first assertions of Proposition \ref{c0prop1} and Proposition \ref{c0prop3}. Thus $\rho,p$ exists globally. 

{\bf Case 3:} 
$$
A_0 \in \left( 0, \frac{k}{N-2}\left( \left( \frac{-\beta y_{\mathfrak{M},N}}{y} \right)^{1-\frac{2}{N}} -1 \right) \right).
$$
Using \eqref{c0smaxexp} with $y_{\mathfrak{M},N}=s_{max}$ and rearranging gives,
$$
\kappa \in \left( 1, \left( \frac{s_{max}}{s_0} \right)^{\frac{1}{N}} \right),
$$
with $\kappa$ as defined in \eqref{kappaexp}. Note that the quantity $y^M$, which 
originally appeared implicitly in the definition of $\Sigma_N$, is now given explicitly by \eqref{c0smaxexp}. In particular, it depends only on $q_0$ and $s_0$.  Since $A_0>0$, formula \eqref{kappaexp} directly shows that $\kappa>1$. 

From condition \eqref{apos1}, we obtain,
$$
\eta_1(0)< \eta_0 < \eta_2(0),
$$
when $q_0 < 0$, and no additional restrictions when $q_0>0$. Note that $q_0=0$ is incompatible with $A_0>0$. Therefore, global-in-time existence of $\rho,p$ follows from the first assertion of Proposition \ref{c0prop3} (for $q_0<0$) and the first assertion of Proposition \ref{c0prop1} (for $q_0>0$).

%Lastly, we assume 
{\bf Case 4:} \\
$$
A_0 \geq  \frac{k}{N-2}\left( \left( \frac{-\beta y_{\mathfrak{M},N}}{y} \right)^{1-\frac{2}{N}} -1 \right) .
$$
As in the previous case, this inequality is equivalent to 
$$
\kappa\geq \left(s_{max}/s_0\right)^{\frac{1}{N}}.
$$
Condition \eqref{apos2} implies $q_0>0$, and the first assertion of Proposition \ref{c0prop1} then gives global-in-time existence of $\rho,p$. 

Collecting all cases and applying Proposition \ref{rhoptogether}, we conclude  the solutions to \eqref{fullodesys2} exist globally as long as each characteristic curve satisfies the conditions obtained above.  In particular, when all  characteristics satisfy these conditions, Lemma \ref{radialfield} and Theorem \ref{local} together yield global-in-time solutions to \eqref{mainsys} with $c=0$. 

Conversely, suppose that there exists a characteristic path corresponding to some parameter $\beta^\ast >0$ such that,
\begin{align*}
(\beta^\ast, u_0^\ast,\phi_{0r}^\ast,u_{0r}^\ast,\rho_0^\ast) & :=(\beta^\ast,u_0(\beta^\ast),\phi_{0r}(\beta^\ast),u_{0r}(\beta^\ast),\rho_0(\beta^\ast))\\
& \notin\Sigma_N\cup \{ (\beta,x,y,z,0): x>0, z\geq -ky/x \}.
\end{align*}
If $\rho_0^\ast = 0$,  Proposition \ref{proprho0} immediately gives finite-time blowup of $p$. 

If instead  $\rho_0^\ast>0$, then it is possible that 
$$
A_0^\ast\leq -k/(N-2),
$$
in which case finite-time breakdown follows directly from Propositions \ref{c0ftb} or \ref{c0prop4}. 

If $A_0^\ast> -k/(N-2)$, then the negation of one of the conditions \eqref{aneg}-\eqref{apos2} must hold, 
%among \eqref{aneg}, \eqref{azero}, \eqref{apos1}, \eqref{apos2} has to be true as and according to 
depending on the value of $A_0^\ast$. Once again, we can check each condition one by one. The analysis hen proceeds exactly as above, except that we invoke the finite-time breakdown conclusions of Propositions \ref{c0prop1}, \ref{c0prop2} and \ref{c0prop3} instead of the global existence assertions. Consequently, solutions to \eqref{mainsys} fails to remain smooth.

This completes the proof of the theorem. \qed

Theorem \ref{c0ctc2} can be proved in essentially the same manner as Theorem \ref{c0ctcn} except that instead of Propositions \ref{c0prop1}, \ref{c0prop2} and \ref{c0prop3}, we use Propositions \ref{c0n2prop1}, \ref{c0n2prop2} and \ref{c0n2prop3}.

\section{Final remarks}
\label{secconc}
The techniques developed in this paper also apply to one-dimensional setting. However, as emphasized in \cite{Tan21}, 
there is a fundamental difference between the 1D and multi-D cases: in multiple dimensions the Poisson forcing alone prevents concentration of the flow at the origin.  In fact, by Corollary \ref{qsalltime},  no matter how large (in absolute value) the initial velocity is, there are no concentrations at the origin for $N\geq 2$. 

To contrast this with the one-dimensional case, consider system
\eqref{qode}-\eqref{sode} with $N=1$, 
$$
q' = k\tilde s - kc - q^2,\qquad \tilde s' = -q\tilde s,
$$
where $\tilde s:= s+c$. Introducing the variables,  $a=q/\tilde s, b = 1/\tilde s$, we obtain the linear system,
$$
a' = k-kcb,\qquad b' = a.
$$
This system can be solved explicitly: 
$$
b(t) = \frac{1}{c} + \left( b(0)-\frac{1}{c} \right)\cos(\sqrt{kc}t) + \frac{a(0)}{\sqrt{kc}}\sin(\sqrt{kc}t).
$$
From this expression, one directly concludes that $b(t_\ast)=0$ (or equivalently, $\lim_{t\to t_\ast^-}\tilde s(t) = \infty$) for some $t_\ast>0$ whenever 
$$
a(0)^2\geq 2kb(0)-kc(b(0))^2,
$$
which in the original variables is equivalent to,
$$
|q|> \sqrt{k(2\tilde s(0)-c)}.
$$
Hence,  sufficiently large initial velocity (in absolute value) leads to concentration at the origin. When $c=0$, the condition becomes one sided:  a sufficiently large negative initial velocity (corresponding to flow pointing towards origin) triggers concentration. 

This behavior is in stark contrast with the case $N\geq 2$, wherein the Poisson forcing completely suppresses such concentrations. Moreover, recent works \cite{CS23,R23} show  that the four dimensional case is quite different from the others. In the present paper, we make an important further observation: for $N=4$, the subcritical region  contains 
initial data with arbitrarily large velocities.  In other words, no matter how large the initial velocity is (positive or negative), there always exists a region in the phase plane of the initial density and gradient of velocity guaranteeing global-in-time smooth solutions. 

Finally, we expect that the techniques introduced here can be extended to other models, such as the Euler–Poisson–alignment system and the Euler–Poisson equations with swirl. We leave these investigations for future work.

\section*{Acknowledgments}
This research was partially supported by the National Science Foundation under Grant DMS1812666.

\bigskip

\bibliographystyle{abbrv}

\end{document}